\documentclass[11pt]{amsart}
\usepackage[utf8]{inputenc}
\usepackage{a4wide}
\usepackage{amsmath,mathrsfs,eucal}
\usepackage{xcolor}
\usepackage{bbm,stix}
\usepackage{enumitem}
\usepackage{graphics} 
\usepackage{epsfig} 
\usepackage{graphicx,subfig}                     
\usepackage{wrapfig}
\usepackage{textcomp}
\usepackage{tikz}
\usepackage{pgfplots}
\usepackage{dsfont}
\pgfplotsset{compat=1.16}
\usepackage{comment}
\usepackage{multicol}  
\usepackage{epstopdf}
\usepackage[colorlinks=true,allcolors=black]{hyperref}
\usepackage{bm}                        \usepackage[]{appendix}  
\usepackage{booktabs}                                                                          
\usepackage{graphicx,subfig}                                                       \usepackage{float} 
\usepackage{wrapfig}              

\newtheorem{defn}{Definition}[section]
\newtheorem{thm}{Theorem}[section]
\newtheorem{prop}{Proposition}[section]
\newtheorem{lem}{Lemma}[section]
\newtheorem{cor}{Corollary}[section]

\newcommand{\N}{\mathbb{N}}
\newcommand{\R}{\mathbb{R}}

\newcommand{\ep}{{\epsilon}}
\newcommand{\me}{{m_\ep}}
\newcommand{\ce}{{c_\ep}}
\newcommand{\de}{{d_\ep}}
\newcommand{\Te}{{T_{max}}}

\newcommand{\Li}{{L^{\infty}(\Omega)}}
\newcommand{\Luno}{{L^{1}(\Omega)}}

\newcommand{\Wi}{{W^{1,\infty}(\Omega)}}
\newcommand{\into}{{\int_{\Omega}}}

\newcommand{\Tin}{{(0, \Te)}}
\newcommand{\medot}{{\me(\cdot,t)}}

\newcommand{\dx}{{\mathrm{d}x}}
\newcommand{\dt}{{\mathrm{d}t}}

\newcommand{\dd}{{\mathrm{d}}}

\begin{document}
\author{S. Fagioli \and E. Radici \and L. Romagnoli}
\address{ Simone Fagioli - DISIM - Department of Information Engineering, Computer Science and Mathematics, University of L'Aquila, Via Vetoio 1 (Coppito)
67100 L'Aquila (AQ) - Italy}
\email{simone.fagioli@univaq.it}
\address{ Emanuela Radici - DISIM - Department of Information Engineering, Computer Science and Mathematics, University of L'Aquila, Via Vetoio 1 (Coppito)
67100 L'Aquila (AQ) - Italy}
\email{emanuela.radici@univaq.it}
\address{ Licia Romagnoli - Dipartimento di Matematica e Fisica - Facoltà di Scienze Matematiche, Fisiche e Naturali,
Università Cattolica del Sacro Cuore di Brescia,
Sede del Buon Pastore
Via Musei 41 
25121 Brescia (BS) - Italy}
\email{licia.romagnoli1@unicatt.it}

\title[On a chemotaxis model with nonlinear diffusion  modelling multiple sclerosis]{On a chemotaxis model with nonlinear diffusion  modelling multiple sclerosis}
\date{}
\begin{abstract}
We investigated existence of global weak solutions for a system of chemotaxis-hapotaxis type with nonlinear degenerate diffusion arising in modelling multiple sclerosis disease. The model consists of three equations describing the evolution of macrophages ($m$), cytokine ($c$) and apoptotic oligodendrocytes ($d$) densities. 
The main novelty in our work is the presence of a nonlinear diffusivity $D(m)$, which results to be more appropriate from the modelling point of view. Under suitable assumptions and for sufficiently regular initial data, adapting the strategy in \cite{li2016,tao2011}, we show the existence of global bounded solutions for the model analysed. 
\end{abstract}
\maketitle

\section{Introduction}

Multiple Sclerosis (MS) is a chronic inflammatory disease that affects the central nervous system, including the brain and spinal cord. It can lead to progressive disability. The disease is caused by an abnormal response of the immune system, resulting in inflammation and damage to neurons and myelin surroundig them. Myelin is a lipid-rich material that surrounds the axons of neurons and facilitates the transmission of nerve impulses. It is produced by the oligodendrocytes. The immune system, specifically the macrophages, attacks and destroys the oligodendrocytes and the myelin sheath around the nerves. This demyelination process leads to the formation of lesions (referred to as plaques in 2D sections) in the white matter of the brain \cite{Lass18}.

Demyelination in  (MS) patients is a heterogeneous process that gives rise to various clinical variants. In the classical study \cite{Lucchinettietal00}, four different types of lesions (Type I - IV) were identified, each corresponding to a different clinical variant. The presence of different lesion types reflects the stage-dependent nature of the pathology, with the evolution of lesional pathology contributing to the heterogeneity. Specifically, Type III lesions are typical in the early stages of the disease, followed by Type I and Type II lesions \cite{Barnett09}. In type II lesions, it is suggested that activated T-lymphocytes may be responsible for inducing macrophage activation. On the other hand, in type III lesions, there is a prominent activation of macrophages observed, accompanied by relatively mild infiltration of T-cells \cite{barnett,marik}. These observations highlight the heterogeneity in the underlying mechanisms and cellular interactions involved in different types of MS lesions. It is also hypothesised that the activation of microglia, a type of macrophages in the central nervous system, plays a role in the development of early multiple sclerosis (MS) lesions \cite{ponomarev}. However, the exact underlying mechanism behind this activation is still unknown.

One clinical variant of particular interest is Baló's concentric sclerosis, which is a rare and aggressive form of the disease described in the literature as a form of MS that exhibits an acute fulminant disease course, leading to rapid progression and death within a few months. Baló MS is characterised by the presence of large demyelinated lesions displaying a distinct pattern of concentric layers, alternating between areas of preserved and destroyed myelin. These lesions are classified as Type III lesions according to the classification system mentioned earlier \cite{Balo}.





In \cite{khonsari} the authors proposed a reaction-diffusion-chemotaxis model to capture the dynamics of early-stage multiple sclerosis, with a specific focus on describing Baló's sclerosis. The model consists of three equations that govern the evolution of activated macrophages, cytokines, and apoptotic oligodendrocytes. We introduce below the model referring to \cite{lombardo} for a more detailed description.

The evolution of macrophages in the model is influenced by three mechanisms. Firstly, macrophages undergo random movement, which is described by a linear isotropic diffusion term. Secondly, they exhibit chemotactic motion in response to a chemical gradient provided by the cytokines. These various factors contribute to the evolution of activated macrophages, denoted by the variable $m$ in the model. The above considerations motivate the following evolution of the activated macrophages $m$
\[
\partial_t m = D\Delta m - \chi \nabla \cdot \left(\frac{m}{1+\delta m}\nabla c\right)+ m\left(1-m\right).
\]
The reaction term in the equation mentioned above captures the production and saturation of activated macrophages. The coefficient $\chi$ represents the chemotactic sensitivity, indicating how sensitive the macrophages are to the chemical gradient provided by the pro-inflammatory cytokines.

The pro-inflammatory cytokines, which are signaling molecules involved in the immune response, are assumed to be produced by both the damaged oligodendrocytes and activated macrophages. In the model, they are described by an equation that takes into account their linear diffusion (possibly occurring at a different scale compared to the macrophages) and degradation. The equation governing the evolution of pro-inflammatory cytokines $c$ can be written as follows:
\[
\tau\partial_t c = \alpha\Delta c -c+ \lambda d  +\beta m.
\]
The parameters $\tau$ and $\alpha$ are positive constants, and $\lambda$, $\beta$, and $\delta$ are non-negative constants.

The destroyed oligodendrocytes, which are the target cells of the immune response, are assumed to be immotile. As a result, there is no spatial dynamics associated with them, and their evolution is governed by the following equation:
\[
\partial_t d = rm \frac{m}{1+\delta m}\left(1-d\right).
\] 
The parameter $r$ in the equation balances the speed of the front and the intensity of the macrophages in damaging the myelin. It determines the relative contribution of the damaging term $\frac{m}{1+\delta m}$, which has been chosen to be positive and increasing with saturation for high values of the macrophages density.

It is worth noting that when $\tau=0$, the model \eqref{eq:main_PM} reduces to a parabolic-elliptic chemotaxis system with a volume-filling effect and a logistic source.

The system introduced in \cite{calvez,khonsari} in the above form from \cite{lombardo} has been studied in recent years through several contributions. Authors in \cite{barresi16,bilotta19,lombardo} investigated various issues related to the structure, stability of stationary states, radial solutions and travelling waves. Furthermore, see also \cite{bisi2023chemotaxis} for a different term describing the production and saturation of activated macrophages. The global existence of strong solutions to this system was proven in one dimension in \cite{DevGiu21} and later extended to any dimension in \cite{desvillettes_giunta_morgan_tang_2022,HU20206875}, where it was also shown that the solution remains uniformly bounded in time. Additionally, we would like to mention the extension of the model that was introduced in \cite{MoFr21}, where the authors introduced a multi-species system to describe the activity of various pro- and anti-inflammatory cells and cytokines in the plaque, and quantified their effect on plaque growth.

The above-mentioned model considered cell random motility, denoted by $D$, as a constant, resulting in linear isotropic diffusion. As emphasised in the classical references \cite{SzMRLaCh09,tao2011}, from a physical perspective, cell migration through biological tissues can be modelled as movement in a porous medium. Thus, we are led to consider the cell motility $D$ as a \emph{nonlinear} function of the macrophage density, denoted as $D(m)$. Several possible choices have been presented in the literature to model different types of movement, including volume filling effects and saturation. However, in the present work, we will focus on power-law type nonlinearities for $D(m)$, specifically $D(m)\cong m^{\gamma-1}$, where $\gamma>1$. More precisely, in this paper, we will investigate the following generalization of the model introduced in \cite{calvez,khonsari,lombardo}:
\begin{equation}\label{eq:main_PM}
    \begin{cases}
        \partial_t m = \nabla \cdot \left(D(m) \nabla m\right) - \chi \nabla \cdot \left(f(m) \nabla c\right) + M(m), \\
        \tau \partial_t c = \alpha \Delta c + \lambda d - c + \beta m, \\
        \partial_t d = r h(m) \left(1-d\right),
    \end{cases}
\end{equation}
where the system is posed in a bounded domain $\Omega \subset \mathbb{R}^n$ with smooth boundary $\partial\Omega$. Our approach will be based on the strategy presented in \cite{li2016}, where the existence of global classical solutions is established through a regularisation argument on the degenerate diffusivity.

Since Keller and Segel \cite{keller} introduced the classical chemotaxis model in 1970, the Keller-Segel model and its modified versions have been widely studied by many researchers over the years. References such as \cite{bellomo2015,bellomo22} provide an extensive overview of these studies. It is well known that the formation of a cell aggregate can lead to finite-time blow-up phenomena. For instance, for the Keller-Segel model in the following form:
\begin{displaymath}
\begin{cases}
u_t=\Delta u-\nabla\cdot(u\cdot\nabla v)\\
\tau v_t-\Delta v+v=u,
\end{cases}
\end{displaymath}
researchers have investigated global solutions and blow-up solutions, as documented in \cite{bellomo2015,jin,winkler2013}.

Due to presence of the third equation and the coupling in the second equation, system \eqref{eq:main_PM} can be regarded as a chemotaxis-haptotaxis model, with the first such model introduced in \cite{chap} to describe the invasion process of cancer cells into surrounding normal tissue. In this models, the random diffusion of activated macrophages is characterised by linear isotropic diffusion, which corresponds to the model \eqref{eq:main_PM} with $\gamma=1$. For this specific case, global classical solutions have been obtained in two-dimensional spatial domain by \cite{tao}, while for the three-dimensional case, global classical solutions are only obtained for large values of $\frac{\mu}{\chi}$, as shown in \cite{cao}. In \cite{tao2011}, the authors considered the case of nonlinear diffusion without degeneracy, where the standard porous medium diffusivity $\Delta m^{\gamma}$ is replaced by $\Delta((m+\epsilon)^{\gamma-1}m)$. Global classical solutions are established for this case, subject to certain restrictions on the possible porous medium exponents related to the problem dimension.

Further results in this direction have been obtained in \cite{li2016} and \cite{wang}, where the existence of global and bounded classical solutions is shown for any $\gamma > \frac{2n-2}{n}$. Recently, Zheng \cite{zheng} extended these results to the cases $\gamma > \frac{2n}{n+2}$. However, the cases $1<\gamma\leq\frac{2n}{n+2}$ remain unknown. On the other hand, for the fast diffusion cases, i.e., $0<\gamma<1$, to the best of our knowledge, there have been no relevant studies. In \cite{liu2020}, the author presents further improvements in the optimality conditions for the nonlinear diffusion exponent.


\subsection*{Structure of the paper}
The paper is organised as follows. In Sect.\,\ref{notazione} we clarify the notation and we list all the assumptions. In Sect.\,\ref{mainresults} are stated the main results of this paper.  In Sect.\,\ref{reg} a detailed investigation of the regularised problem associated to \eqref{eq:main_PM} is performed in order to show some useful a-priori estimates and boundedness of the solutions. The local existence in time of the weak solutions to the regularised problem is shown in the Appendix \ref{sec:appendix united}.
In Sect.\ref{sec:nonregularised} we derive the existence of global weak solutions of \eqref{eq:main_PM} by showing suitable compactness of the solutions of the regularised problems. Several known technical results which are used in the proof are collected in Appendix \ref{pre}. Final remarks on the fundamental role of the nonlinear diffusion adopted in model \eqref{eq:main_PM}, endorsed by some numerical simulations on $1D$ and $2D$ domains, are listed in Sect.\ref{sec:conclusion} together with possible developments of the problem that are not further investigated in the present work. 
In  Appendix \ref{sec:appendix num} we describe the employed finite volume numerical scheme.

\section{Preliminaries and Main Result}\label{boundedness}
\subsection{Assumptions}\label{notazione}
Let $\Omega \subset \R^n$, with $n=1,2,3$, be a bounded domain with smooth boundary and $t>0$. Without loss of generality we fix the constants $\tau=\alpha=\lambda=\beta=r=1$ and we consider the following form for system \eqref{eq:main_PM}
\begin{equation}\label{eq:main_Gen}
    \begin{cases}
\partial_t m = \nabla\cdot \left( D(m) \nabla m - \chi  f(m)\nabla c\right)+M(m), \\
\partial_t c = \Delta c  - c +d+ m, \\
\partial_t d = h(m)\left(1-d\right),
\end{cases}\quad x\in \Omega,\,t>0,
\end{equation}
endowed with the Neumann boundary conditions
\begin{equation}\label{neumbc}
\frac{\partial m}{\partial\mathbf{n}}\bigg|_{\partial\Omega}=\frac{\partial c}{\partial\mathbf{n}}\bigg|_{\partial\Omega}=0,
\end{equation}
and subject to the initial conditions
\[
 m(x,0)=m_0(x),\, c(x,0)=c_0(x),\,d(x,0)=d_0(x),
\]
where the functions $m_0$, $c_0$ and $d_0$ satisfy a standard compatibility condition in the sense that
    \begin{equation}\label{indata}
     \begin{cases}
    m_0\in L^{\infty}(\Omega),\,\nabla m_0\in L^{2}(\Omega)\ \ \mbox{with}\ \ m_0\geq0\ \ \mbox{in} \ \ \Omega \ \ \mbox{and} \ \ m_0\neq0, \\
    c_0\in W^{1,\infty}(\Omega)\ \ \mbox{with}\ \ c_0\geq0\ \ \mbox{in} \ \ \Omega \ \ \mbox{and} \ \ \displaystyle{\frac{\partial c_0}{\partial\mathbf{\nu}}}=0 \ \ \mbox{on} \ \ \partial\Omega, \\
    d_0\in C^{2}(\bar{\Omega})\ \ \mbox{with}\ \ 0\leq  d_0\leq 1\ \ \mbox{in} \ \ \bar{\Omega} \ \ \mbox{and} \ \ \displaystyle{\frac{\partial d_0}{\partial\mathbf{\nu}}}=0 \ \ \mbox{on} \ \ \partial\Omega,  
    \end{cases}
\end{equation}

In system \eqref{eq:main_Gen} we consider a nonlinear diffusion function $D$ under the following assumptions
\begin{equation}\tag{A-D1}\label{AssD1}
    D\in C^2([0,\infty)),\,D(s)\geq 0,\,\,\mbox{for all}\,s\geq0, \mbox{ and }\,
\end{equation}
\begin{equation}\tag{A-D2}\label{AssD2}
\mbox{there exist some constants $\gamma > 1$ and $k_D>0$ such that }\,
D(s)\geq k_D s^{\gamma-1}\,\mbox{for all}\,s\geq0,
\end{equation}
\begin{equation}\tag{A-D3}\label{ass_D}
 \mbox{we denote by $\Phi$ the primitive function of $D$, namely }\, \Phi(u)=\int_0^u D(s)\, \dd s.
\end{equation}
Concerning the reaction term in the the first equation of \eqref{eq:main_Gen} we assume that
\begin{equation}\tag{A-M}\label{ass_M}
M\in C^1([0,\infty)),\,\mbox{and exists $\mu>1$ such that }\,
\frac{M(s)}{s}\leq\mu(1-s)\,\mbox{for all}\,s>0.   
\end{equation}
We finally assume that
\begin{equation}\tag{A-h}\label{ass_h}
h\in C^1([0,\infty)),\,\mbox{non-negative and exists $k_h>0$ such that }\,    h'(s)\leq k_h,\,\mbox{for all}\,s\geq0. 
\end{equation}
and
\begin{equation}\tag{A-f}\label{ass_f}
f\in C^1([0,\infty)),\,\mbox{such that $f(0)=0$ and exists $k_f>0$ such that }\,   |f(s)|\leq k_f s ,\,\mbox{for all}\,s\geq0. 
\end{equation}

\subsection{Main result}\label{mainresults}
We now state the notion of solutions we are dealing with
\begin{defn}\label{def:weak_sol}
Let $T>0$ be a real number. Consider $\Omega\subset\mathbb{R}^n$ a bounded domain with smooth boundary $\partial\Omega$. A triple $(m, c, d)$ of non-negative functions is called a weak solution of \eqref{eq:main_Gen} endowed with the Neumann boundary conditions
\eqref{neumbc} on $\Omega\times\left[0,T\right)$ and initial conditions $\left(m_0,c_0,d_0\right)$, if 
\begin{align*}
     (m, c, d)\in L_{loc}^2([0,T);L^2(\Omega))\,\times\, L_{loc}^2([0,&T);W^{1,2}(\Omega))\, \times\, L_{loc}^2([0,T);W^{1,2}(\Omega)),\\
     D(m)\nabla m\in &\,L_{loc}^2([0,T);L^2(\Omega)),
\end{align*}
and for any $\phi\in C_0^{\infty}\left(\overline{\Omega}\times \left[0,T\right)\right)$
\begin{equation}\label{weak_m}
    \begin{aligned}
-\int_0^T\into m\phi_t\, \dx\dt-\int_{\Omega} m_0\phi(x,0)\,\dx=&-\int_0^T\into D(m)\nabla m\nabla\phi \,\dx\dt\\
&+\chi\int_0^T\into f(m)\nabla c\nabla\phi \,\dx\dt+\int_0^T\into M(m)\phi \,\dx\dt,
\end{aligned}
\end{equation}
\begin{equation}\label{weak_c}
  -\int_0^T\into c\phi_t\,\dx\dt-\int_{\Omega}c_0\phi(x,0)\,\dx=-\int_0^T\into\nabla c\nabla\phi\,\dx\dt+\int_0^T\into(m+d-c)\phi\,\dx\dt,
\end{equation}
\begin{equation}\label{weak_d}
    -\int_0^T\into d\phi_t\,\dx\dt-\int_{\Omega}d_0\phi(x,0)\dx=\int_0^T\into mh(m)(1-d)\phi\,\dx\dt.
\end{equation}
If the triple $(m, c, d)$ satisfies the conditions for all $T>0$, the weak solution is called global.
\end{defn}

In the following we will assume that the diffusion exponent $\gamma$ satisfies the following restrictions
\begin{equation}\tag{cond-$\gamma$}\label{cod_gamma}
 \gamma>\max\left\{2-\frac{2}{n},1\right\}\, \mbox{ for }\, n=1,2,3.
\end{equation}
The main result of the paper is contained in the following 
\begin{thm}\label{thm1}
Let $\Omega\subset\mathbb{R}^n$ be a bounded domain with smooth boundary $\partial\Omega$ and $\chi,\mu>0$. Assume that $D$ satisfies \eqref{AssD1}-\eqref{AssD2}-\eqref{ass_D}, $M$, $f$ and $h$ are under assumptions \eqref{ass_M}, \eqref{ass_f} and \eqref{ass_h} respectively. Consider a triple $\left(m_0,c_0,d_0\right)$ satisfying \eqref{indata}. Then, for any $\gamma>1$ that satisfies \eqref{cod_gamma}, there is $C>0$ such that system \eqref{eq:main_Gen} endowed with the Neumann boundary conditions
\eqref{neumbc} has a weak solution $\left(m,c,d\right)$ in the sense of Definition \ref{def:weak_sol} that exists globally in time and satisfies
\[
\|m(\cdot,t)\|_{L^\infty(\Omega)}+\|c(\cdot,t)\|_{W^{1,\infty}(\Omega)}+\|d(\cdot,t)\|_{W^{1,\infty}(\Omega)}\leq C,
\]
for all $t\in (0,\infty)$.
\end{thm}

\section{Regularised non-degenerate system}\label{reg}
Existence and boundedness of global weak solutions to system \eqref{eq:main_Gen} will be proved by introducing a regularised (non-degenerate) problem for which we are able to construct a global classical solution. Moreover, the regularised system admits bounds that are independent of the regularisation parameter. This later allows to pass to the limit and obtain a solution to the original degenerate problem.
In order to do so, for $\ep \in (0, 1)$, we introduce the function $D_\ep$ defined by
\begin{equation}\label{depsilon}
    D_\ep (s) = D(s+\ep).
\end{equation}
Note that according to Assumptions \eqref{AssD1}-\eqref{ass_D} we have that $D_\ep (s) \geq k_D (s+\ep)^{\gamma-1}$ and $D_\ep(0)>0$. Moreover, we can introduce the primitive of $D_\ep$:
\begin{equation}\label{phiepsilon}
    \Phi_\epsilon (s) =\int_{0}^{s} D_{\ep}(\sigma) \dd \sigma =\Phi(s+\ep)-\Phi(\ep).
\end{equation}
Given a triple $\left(m_0,c_0,d_0\right)$ satisfying \eqref{indata} we introduce $\left(\me_0,\ce_0,\de_0\right)$ such that,  for some $\vartheta\in(0,1)$,
\begin{equation}\label{indataregul}
\me_0, \ce_0, \de_0\in C^{2+\vartheta}(\overline{\Omega}),\, 0\leq \de_0\leq 1, \,\mbox{  and }\,
\frac{\partial\me_0}{\partial\mathbf{n}}\bigg|_{\partial\Omega}=\frac{\partial\ce_0}{\partial\mathbf{n}}\bigg|_{\partial\Omega}=0, 
\end{equation}
and 
\begin{equation}\label{indataregul_conv}
\me_0\rightarrow m_0, \ce_0\rightarrow c_0, \de_0\rightarrow d_0,\,\mbox{ in $C(\bar{\Omega})$ and }\, \nabla\ce_0\rightarrow \nabla c_0, \nabla\de_0\rightarrow \nabla d_0 \mbox{ in }\, L^2(\Omega).
\end{equation}

For $T>0$, we consider the following regularised system
\begin{subequations}
\label{eq:regulare}
\begin{align}
\label{eq:regulare_m}
&\partial_t \me  = \nabla\cdot\left(D_\ep(\me)\nabla\me\right) - \chi \nabla \cdot \left(f(\me)\nabla \ce\right)+M(\me), \\
\label{eq:regulare_c}
&\partial_t \ce  = \Delta \ce + \de - \ce + \me, \\
\label{eq:regulare_d}
&\partial_t \de  = h(\me)\left(1-\de\right),
\end{align}
\end{subequations}
with $(x,t)\in\Omega_{T}$, under the initial conditions
\begin{equation}\label{in_reg}
    \me(x,0)=\me_0(x),\, \ce(x,0)=\ce_0(x),\, \de(x,0)=\de_0(x).
\end{equation}
System \eqref{eq:regulare} is endowed with the Neumann boundary conditions
\begin{equation}\label{neumbcregulare}
\frac{\partial\me}{\partial\mathbf{n}}\bigg|_{\partial\Omega}=\frac{\partial\ce}{\partial\mathbf{n}}\bigg|_{\partial\Omega}=0.
\end{equation}

Before going further, we first state a result on the local existence in time of classical solutions to \eqref{eq:regulare}, which can be attained by employing well-known fixed point arguments and standard parabolic regularity, see for instance \cite{tao2008,winkler2008,taowin11} for similar results in different type of models. The proof can be found in Appendix \ref{sec:appendix united}.

\begin{prop}\label{locexistregulare}
Let $\ep\in(0,1)$ and $\chi>0$. Assume that the non-negative functions $\me_0, \ce_0$ and $\de_0$ satisfy \eqref{indataregul} for some $\vartheta\in(0,1)$. Consider $D_\ep$ as in \eqref{depsilon}. Then there exists a maximal existence time $\Te\in(0, \infty]$ and a triple of non-negative functions 
\begin{align*}
\me\in C^0(\bar{\Omega}\times[0,\Te))&\cap C^{2,1}(\bar{\Omega}\times(0, \Te)),\notag \\
\ce\in C^0(\bar{\Omega}\times[0,\Te))&\cap C^{2,1}(\bar{\Omega}\times(0, \Te)),\notag\\
\de\in C^{2,1}(\bar{\Omega}&\times(0, \Te))
\end{align*}
that solves \eqref{eq:regulare} classically on $\Omega \times (0,\Te)$ and satisfies $0\leq \de\leq\|\de_0\|_{L^{\infty}(\Omega)}$, $\me\geq0$ and $\ce\geq 0$ in $\Omega\times(0, \Te)$. Moreover, either $\Te = +\infty$, or 
\begin{equation}\label{localtoinf}
    \|\me(\cdot, t)\|_{\Li}+\|\ce(\cdot, t)\|_{\Wi}+\|\de(\cdot, t)\|_{\Li}\rightarrow\infty \ \ \mbox{as} \ \ t\nearrow\Te.
\end{equation}
\end{prop}
According to the above existence theory, for any $s\in(0,\Te)$, $(\me(\cdot, s),\ce(\cdot, s),\de(\cdot, s))\in C^2(\bar{\Omega})$. We can assume that there exists a positive constant $K$ such that
\begin{equation}
    \|\me_0\|_{C^2(\bar{\Omega})}+\|\ce_0\|_{C^2(\bar{\Omega})}+\|\de_0\|_{C^2(\bar{\Omega})}\leq K.
\end{equation}

\subsection{A-priori estimates}
In this section, we establish a series of uniform in $\ep$ a-priori estimates for solutions of system \eqref{eq:regulare}. Firstly, we provide a bound on the total mass of $\me$.

\begin{lem}\label{boundlem}
There exists a positive constant $K_0$ only depending on $\Omega$ and $ \|m_{\ep0}\|_{\Luno}$ such that solutions $(\me, \ce, \de)$ of \eqref{eq:regulare} satisfies 
\begin{equation}\label{es:l1me}
    \|\me(\cdot, t)\|_{\Luno}\leq K_0\ \ \mbox{for all} \ \ t\in\Tin.
\end{equation}
Moreover, if $M=0$
\begin{equation}\label{es:l1me_con}
    \|\medot\|_{L^1(\Omega)}=\|\me_0\|_{L^1(\Omega)} \ \ \mbox{for all} \ \ t\in\Tin.
\end{equation}
\end{lem}
\begin{proof}
The regularity stated in Proposition \ref{locexistregulare}, together with the Neumann boundary conditions in \eqref{neumbcregulare} allow a direct integration of \eqref{eq:regulare_m}. Then, \eqref{es:l1me} follows by an ODE comparison argument. Finally, if the reaction term $M$ is zero, then the equation for $\me$ is in divergence form, then conservation of mass is a straightforward consequence of the previous integration.
\end{proof}

We start showing that solutions to \eqref{eq:regulare_d} remain bounded.

\begin{lem}\label{de_sign}
For any $p\geq1$ the solution $(\me, \ce, \de)$ to \eqref{eq:regulare} satisfies
\begin{equation}\label{stima_sign}
   \de\leq 1, \mbox{ for all } t\in (0, \Te),\,x\in\Omega.
\end{equation}
\end{lem}
\begin{proof} Consider $\eta_{-,\delta}$ a smooth approximation of the negative part function, for some $\delta >0$. Then,
\begin{align*}
    \frac{\dd}{\dt}\into \eta_{-,\delta}\left(1-\de\right)\dx&=-\into \eta'_{-,\delta}\left(1-\de\right)\partial_t\de\dx\\
    &=-\into \eta'_{-,\delta}\left(1-\de\right)h(\me)\left(1-\de\right)\dx\leq 0,
\end{align*}
thus, $1-\de$ maintains its sign along the evolution, uniformly respect to the regularisation.
\end{proof}

By multiplying equation \eqref{eq:regulare_d} by $\de^{p-1}$, using assumptions on $h$ in \eqref{ass_h} and Young's inequality we are able to produce the following estimate.
\begin{lem}\label{de1}
For any $p\geq1$ the solution $(\me, \ce, \de)$ to \eqref{eq:regulare} satisfies
\begin{equation}\label{stima_0}
     \frac{1}{p}\frac{\dd}{\dt}\into \de^p \dx  \leq \frac{k_h^p}{p}\into \me^p\dx + \frac{p-1}{p}\into \de^p \dx,
\end{equation}
for all $t\in (0, \Te)$. In particular
\begin{equation}\label{stima_0_1}
    \|\de\|_{L^1(\Omega)}\leq k_h K_0,
\end{equation}
for all $t\in (0, \Te)$, where $K_0$ is the constant introduced in \eqref{es:l1me}.
\end{lem}

At this point we want to derive a uniform upper bound for $\me$, which represents the key to obtain all the estimates needed. 

\begin{lem}\label{mean1}
Let $\mu,\chi>0$ and $\gamma>1$ be given constants. Then for any $p>1$ the solution $(\me, \ce, \de)$ to \eqref{eq:regulare} satisfies
\begin{align}\label{stima_1}
\begin{aligned}
     \frac{1}{p}\frac{\dd}{\dt}\into \me^p \dx & + \frac{p-1}{2}\into D_\ep(\me)\me^{p-2}|\nabla\me|^2\dx + \frac{k_D(p-1)}{(p+\gamma-1)^2}\into |\nabla \me ^{\frac{p+\gamma-1}{2}}|^2\dx\\
     &\leq \frac{k_f^2}{ k_D}\chi^2(p-1)\into \me^{p-\gamma+1}|\nabla \ce|^2 \dx + \mu \into \me^p\dx - \mu\into \me^{p+1}\dx
\end{aligned}
\end{align}
for all $t\in (0, \Te)$.
\end{lem}
\begin{proof}
We start by multiplying  equation \eqref{eq:regulare_m} by $\me^{p-1}$ and integrating in space where necessary, and obtain
\begin{align}\label{primo_passo}
\begin{aligned}
   &  \frac{1}{p}\frac{\dd}{\dt}\into \me^p \dx  =- (p-1)\into D_\ep(\me)\me^{p-2}|\nabla\me|^2\dx \\
     &-\chi\into \me^{p-1}\nabla\cdot (f(\me)\nabla \ce) \dx +\into \me^{p-1} M_\ep (\me)\dx := I_1 + I_2 + I_3,
\end{aligned}
\end{align}
for $t\in(0, \Te)$. We first estimate the first integral on the r.h.s. $I_1$ as follows
\begin{equation}\label{int_diff}
    I_1 \leq -\frac{p-1}{2}\into D_\ep(\me)\me^{p-2}|\nabla\me|^2\dx-\frac{k_D(p-1)}{2}\into \me^{p-2}(\me+\ep)^{\gamma-1}|\nabla \me|^2\dx,
\end{equation}
where we have used the bound from below on $D_\ep$. By an integration by parts in $I_2$, applying Young's inequality and assumption \eqref{ass_f}, we derive the following estimate
\begin{equation}\label{int_chemo}
    I_2 \leq \frac{k_D(p-1)}{4}\into \me^{\gamma+p-3}|\nabla\me|^2
\dx + \frac{k_f^2}{k_D}\chi^2(p-1)\into \me^{p-\gamma+1}|\nabla \ce|^2 \dx.\end{equation}
Invoking the bound \eqref{ass_M} in order to control $I_3$, and putting together \eqref{primo_passo}, \eqref{int_diff} and \eqref{int_chemo} we easily obtain \eqref{stima_1}.  
\end{proof}

\begin{lem}\label{lem:nabla_c}
Let $\mu,\chi>0$ and $\gamma>1$ be given constants. Then, for any $q\in[1,+\infty)$ the solution $(\me, \ce, \de)$ to \eqref{eq:regulare} satisfies
\begin{align}\label{stima_2}
\begin{aligned}
     \frac{1}{q}\frac{\dd}{\dt}\into |\nabla\ce|^{2q} \dx & + 2 \into |\nabla\ce|^{2q} \dx + \frac{q-1}{q^2}\into |\nabla|\nabla\ce|^q|^2 \dx\\
     &  \leq \left(2(q-1)+\frac{n}{2}\right) \into \left(\de + \me\right)^2|\nabla\ce|^{2q-2} \dx + C_1,
\end{aligned}
\end{align}
for all $t\in (0, \Te)$ and for some positive constant $C_1$ and independent from $\epsilon$.
\end{lem}
\begin{proof}
The proof of estimate \eqref{stima_2} is based on the one in \cite[Lemma 3.3]{li2016}, see also \cite[Proposition 3.2]{ishida2014} and \cite[Lemma 3.3]{tao2012} and it is reported in Appendix \ref{Proofs} for completeness. We observe only that the main difference from \cite[Lemma 3.3]{li2016} is the presence of $\de$ in the equation for $\ce$, which can be managed thanks to \eqref{stima_0_1}.
\end{proof}
Summing up the estimates obtained in Lemmas \ref{mean1} and \ref{lem:nabla_c} we easily obtain the following.
\begin{cor}
Let $\mu,\chi>0$ and $\gamma>1$ be given constants. Then, for any $p>1$ and for any $q\in[1,+\infty)$ the solution $(\me, \ce, \de)$ to \eqref{eq:regulare} satisfies
\begin{align}\label{stima_3}
\begin{aligned}
     \frac{\dd}{\dt}&\left(\into \me^p \dx+ \into |\nabla\ce|^{2q} \dx \right) + \frac{p(p-1)}{2}\into D_\ep(\me)\me^{p-2}|\nabla\me|^2\dx \\
     &+ \frac{k_Dp(p-1)}{(p+\gamma-1)^2}\into |\nabla \me ^{\frac{p+\gamma-1}{2}}|^2\dx+ 2q \into |\nabla\ce|^{2q} \dx + (q-1)\into |\nabla|\nabla\ce|^q|^2 \dx\\
     &\leq  \frac{k_f^2}{k_D}\chi^2 p(p-1)\into \me^{p-\gamma+1}|\nabla \ce|^2 \dx + q\left(2(q-1)+\frac{n}{2}\right) \into \left(\de + \me\right)^2|\nabla\ce|^{2q-2} \dx+C_1, 
\end{aligned}
\end{align}
for all $t\in (0, \Te)$, for a positive constant $C_1$  given in \eqref{stima_2} independent from $\epsilon$ and $p$. 
\end{cor}
The following Lemmas concern with the control of the terms on the r.h.s. of \eqref{stima_3}. The existence of exponents $p$ and $q$ that satisfy the conditions we are going to impose are guaranteed by Lemma \ref{exponents}, under the restriction \eqref{cod_gamma}.
\begin{lem}\label{lemm:primo_incubo}
Assume that $\gamma$ satisfies \eqref{cod_gamma}. Let $p,q>1$ be such that for $n\geq 2$
\begin{equation}\label{condpq_1}
   p>\gamma-\frac{n-2}{nq},\quad
\mbox{ and }\quad
   \frac{q(p-\gamma+1)}{q-1} \left(\frac{n-\frac{nq-n+2}{q(p-\gamma+1)}}{(p+\gamma-1)n+2-n}\right)<1, 
\end{equation}
and, for $n=1$ 
\begin{equation}\label{condpq_2}
 q(p-\gamma)+1>0,\quad \mbox{ and }\quad  \frac{(1+q(p-\gamma))}{(q-1)(p+\gamma)}<1.
\end{equation}
Then for any $\eta>0$ there exists a constant $C>0$ depending only on $\eta$, $p$ and $q$ such that
\begin{equation}\label{stima_4}
    \into \me^{p-\gamma+1}|\nabla\ce|^2\dx \leq \eta \into \bigl|\nabla\me^{\frac{p+\gamma-1}{2}}\bigr|^2\dx+\eta\into \bigl|\nabla|\nabla\ce|^q\bigr|^2\dx+C,
\end{equation}
for all $t\in(0,\Te)$.
\end{lem}
\begin{proof}
The proof is postponed in Appendix \ref{Proofs}.
\end{proof}

\begin{lem}\label{lemm:secondo_incubo}
Assume that $\gamma$ satisfies \eqref{cod_gamma}. Let $p,q>1$ be such that for $n\geq 2$
\begin{equation}\label{condpq_3}
    p>\max\left\{1,\frac{2q(n-2)}{2q+n-2}\right\},\quad \mbox{and}\quad \frac{1}{q}>\frac{2}{p+\gamma-1}\left(\frac{\frac{p+\gamma-1}{2}-\frac{(p+\gamma-1)(2q+n-2)}{4nq}}{\frac{p+\gamma-1}{2}+\frac{1}{n}-\frac{1}{2}}\right),
\end{equation}
and for $n=1$
\begin{equation}\label{condpq_4}
\frac{2q-1}{p+\gamma}<1.
\end{equation}
Then for any $\eta>0$ there exists a constant $C>0$ such that
\begin{equation}\label{stima_5}
    \into \left(\de + \me\right)^2|\nabla\ce|^{2q-2} \dx\leq \eta\bigl\|\nabla\me^{\frac{p+\gamma-1}{2}}\bigr\|_{L^{2}(\Omega)}^2+\eta\bigl\|\nabla|\nabla \ce|^q\bigr\|_{L^{2}(\Omega)}^2+C,
\end{equation}
for all $t\in \Tin$.
\end{lem}

\begin{proof}
The proof is postponed in Appendix \ref{Proofs}.
\end{proof}

\subsection{Boundedness for the regularised system} The bounds gained in Lemmas \ref{lemm:primo_incubo} and \ref{lemm:secondo_incubo} allow to produce the following estimate.
\begin{prop}\label{prop:crucial_estimate}
Let $\mu,\chi\geq 0$ and $\gamma$ under condition \eqref{cod_gamma}. Let $p,q\in (1,\infty)$ under the assumptions of Lemmas \ref{lemm:primo_incubo} and \ref{lemm:secondo_incubo}. Then there exists a constant $C:=C(p,q)>0$ independent from $\ep$ such that
\begin{equation}\label{stima_6}
    \into \me^p\dx +\into |\nabla \ce|^{2q}\dx \leq C,
\end{equation}
for all $t\in \Tin$.
\end{prop}
\begin{proof}
Consider $p$ and $q$ large enough and satisfying the conditions in Lemmas \ref{lemm:primo_incubo} and \ref{lemm:secondo_incubo}. From \eqref{stima_3}, we can set the constant $\eta$ in \eqref{stima_4} and \eqref{stima_5} such that we can deduce the existence of certain constants $k_i$, $i=1,\ldots,4$
\begin{align*}
     \frac{\dd}{\dt}&\left(\into \me^p \dx+ \into |\nabla\ce|^{2q} \dx\right)  + k_1\into |\nabla \me ^{\frac{p+\gamma-1}{2}}|^2\dx+ k_2 \into |\nabla\ce|^{2q} \dx \leq k_3,
\end{align*}
for all $t\in \Tin$. Moreover,
\begin{align*}
    \into \me^p \dx = \bigl\| \me^{\frac{p+\gamma-1}{2}} \bigr\|_{L^{\frac{2p}{p+\gamma-1}}(\Omega)}^\frac{2p}{p+\gamma-1}&\leq k_4 \left(\bigl\| \nabla\me^{\frac{p+\gamma-1}{2}} \bigr\|_{L^2(\Omega)}^\frac{2pa}{p+\gamma-1}\bigl\| \me^{\frac{p+\gamma-1}{2}} \bigr\|_{L^{\frac{2}{p+\gamma-1}}(\Omega)}^\frac{2p(1-a)}{p+\gamma-1}+\bigl\| \me^{\frac{p+\gamma-1}{2}} \bigr\|_{L^{\frac{2}{p+\gamma-1}}(\Omega)}^\frac{2p}{p+\gamma-1}\right)\\
    & \leq k_5 \left(\bigl\| \nabla\me^{\frac{p+\gamma-1}{2}} \bigr\|_{L^2(\Omega)}^\frac{2pa}{p+\gamma-1}K_0^\frac{2p(1-a)}{p+\gamma-1}+K_0^\frac{2p}{p+\gamma-1}\right),
\end{align*}
where $K_0$ is the constant introduced in Lemma \ref{boundlem}. Observing that by construction
\[
\frac{pa}{p+\gamma-1}=\frac{p}{p+\gamma-1}\frac{(p+\gamma-1)(1-\frac{1}{p})}{p+\gamma-1+\frac{2}{n}-1}<1,
\]
we can apply Young's inequality in order to deduce
\[
 \into \me^p \dx \leq \eta  \into |\nabla \me ^{\frac{p+\gamma-1}{2}}|^2\dx +C(\eta),
\]
for any $\eta>0$. Upon a rearrangement in the constants we can combine the above estimates in order to conclude
\begin{align*}
     \frac{\dd}{\dt}&\left(\into \me^p \dx+ \into |\nabla\ce|^{2q} \dx\right) + k_6\left(\into \me^p \dx+ \into |\nabla\ce|^{2q}  \dx\right) \leq k_7,
\end{align*}
and an ODE comparison argument yields \eqref{stima_6}.
\end{proof}
\begin{cor}\label{cor:Lpestimates}
Let $\mu,\chi\geq 0$ and $\gamma$ under condition \eqref{cod_gamma}. Let $p,q\in (1,\infty)$ under the assumptions of Lemmas \ref{lemm:primo_incubo} and \ref{lemm:secondo_incubo}. Then there exists a constant $C:=C(p)>0$ independent from $\ep$ such that
\begin{equation}\label{stima_7}
    \|\me(\cdot,t)\|_{L^p(\Omega)}\leq C,\quad\mbox{and}\quad  \quad  \|\nabla\ce(\cdot,t)\|_{L^p(\Omega)}\leq C ,
\end{equation}
for all $t\in\Tin$.
\end{cor}

In the next Proposition we provide the uniform bounds on regularised solutions $\me$ and $\ce$ that, together with Lemma \ref{de_sign} imply the uniform boundedness of the triple $(\me,\ce,\de)$.
\begin{prop}\label{prop:Linfty}
Let $\mu, \chi>0$ and $\gamma$ under condition \eqref{cod_gamma}. There exists a constant $C>0$ independent from $\ep$ such that
\begin{equation}\label{stima_8}
    \|\me(\cdot,t)\|_{L^\infty(\Omega)}\leq C,\quad  \|\ce(\cdot,t)\|_{W^{1,\infty}(\Omega)}\leq C 
\end{equation}
for all $t\in\Tin$.
\end{prop}
\begin{proof}
Fix $T\in\Tin$. The bound on $\ce$ can be derived in the same spirit of what we did in the proof of Lemma \ref{lem:nabla_c} by using the Duhamel's formula  and $L^p-L^q-$estimates for the Neumann-heat semigroup in the spirit of in \cite[Lemma 1.3]{Winkler2010JDE}, see also \cite[Lemma 2.6]{li2016}. More precisely, for all $t\in (0,T)$ we define
\[
\bar{\me}(t)=\frac{1}{|\Omega|}\into \me(x,t) \dx,\quad\mbox{ and }\quad\bar{\de}(t)=\frac{1}{|\Omega|}\into \de(x,t) \dx,
\]
and, fixing $p>n$,  we can estimate
\begin{align*}
    \|\ce(.,t)\|_{L^\infty(\Omega)}\leq& \|e^{t(\Delta-1)}c_0\|_{L^\infty(\Omega)}+\int_0^t\bigl\|e^{(t-\tau)(\Delta-1)}(\me(\cdot,\tau)+\de(\cdot,\tau))\bigr\|_{L^\infty(\Omega)}d\tau\\
    \leq & \|c_0\|_{L^\infty(\Omega)}+\int_0^t\left[C_1\left((1+(t-\tau)^{-\frac{n}{2p}}\right)e^{-\lambda_1(t-\tau)}\|\me(\cdot,\tau)-\bar{\me}(\tau)\|_{L^p(\Omega)}\right]d\tau\\
    & +\int_0^t\left[C_1\left((1+(t-\tau)^{-\frac{n}{2p}}\right)e^{-\lambda_1(t-\tau)}\|\de(\cdot,\tau)-\bar{\de}(\tau)\|_{L^p(\Omega)}\right]d\tau\\
    & +\int_0^t e^{-(t-\tau)}\left(\|\bar{\me}(\tau)\|_{L^\infty(\Omega)}+\|\bar{\de}(\tau)\|_{L^\infty(\Omega)}\right)d\tau\\
    \leq & \|c_0\|_{L^\infty(\Omega)}+\int_0^t\left[C_1\left((1+(t-\tau)^{-\frac{n}{2p}}\right)e^{-\lambda_1(t-\tau)}4C(p)\right]d\tau\\
    & +\int_0^t e^{-(t-\tau)}|\Omega|^{-1}\left(\|\bar{\me}(\tau)\|_{L^1(\Omega)}+\|\bar{\de}(\tau)\|_{L^1(\Omega)}\right)d\tau\leq C_2,
\end{align*}
since $p>n$, the $L^1-$norms of the average functions are bounded because of Lemmas \ref{boundlem} and \ref{de1}, and $C(p)$ is the constant given by \eqref{stima_7}. Similarly, we can bound
\begin{align*}
    \|\nabla\ce(.,t)\|_{L^\infty(\Omega)}\leq& C_3\|\nabla c_0\|_{L^\infty(\Omega)}\\&+C_4\int_0^t\left[\left((1+(t-\tau)^{-\frac{1}{2}-\frac{n}{2p}}\right)e^{-\lambda_1(t-\tau)}\left(\|\me(\cdot,\tau)\|_{L^p(\Omega)}+\|\de(\cdot,\tau))\|_{L^p(\Omega)}\right)\right]d\tau\\
    & C_3\|\nabla c_0\|_{L^\infty(\Omega)}+C_4\int_0^t\left[\left((1+(t-\tau)^{-\frac{1}{2}-\frac{n}{2p}}\right)e^{-\lambda_1(t-\tau)}2C(p)\right]d\tau\leq C_5.
\end{align*}
In order to derive the estimates for $\me$ we adopt a by now classical iteration procedure on the exponent $p$, see \cite{Ali792,Ali791,li2016,tao2012}. Fix $p_0\in\R$ such that $p_0>\frac{3}{2}(\gamma-1)$ and such that
\[\frac{p(p-1)}{(p+\gamma-1)^2}\in \left(\frac{1}{2},\frac{3}{2}\right), \quad \mbox{ for all}\quad p\in \left[p_0,\infty\right).\]
We first perform the following estimates on $\|\me\|_{L^p(\Omega)}^p$ and $\bigl\|\me^{p+\gamma-1}\bigr\|_{L^1(\Omega)}$. By Young's inequality we have
\begin{equation}\label{est:Lp}
          p(\mu+1)\|\me\|_{L^p(\Omega)}^p \leq \mu p \|\me\|_{L^{p+1}(\Omega)}^{p+1}+\frac{(\mu+1)^{p+1}}{\mu^p}\left(\frac{p}{p+1}\right)^{p+1}|\Omega|,
\end{equation}
while, in combination with Gagliardo-Niremberg inequality, we get
\begin{equation}\label{est:pgammauno}
\begin{aligned}
        \into \me^{p+\gamma-1}\dx = \bigl\|\me^{\frac{p+\gamma-1}{2}}\bigr\|_{L^2(\Omega)}^{2}\leq& C_6 \bigl\|\nabla\me^{\frac{p+\gamma-1}{2}}\bigr\|_{L^2(\Omega)}^{\frac{2n}{n+2}}\bigl\|\me^{\frac{p+\gamma-1}{2}}\bigr\|_{L^1(\Omega)}^{\frac{4}{n+2}}+C_7\bigl\|\me^{\frac{p+\gamma-1}{2}}\bigr\|_{L^1(\Omega)}^{2}\\
        \leq& \frac{nK_1}{n+2} \bigl\|\nabla\me^{\frac{p+\gamma-1}{2}}\bigr\|_{L^2(\Omega)}^{2}+\left(C_6^\frac{n+2}{2}\frac{2K_1^{-\frac{n}{2}}}{n+2}+C_7\right)\bigl\|\me^{\frac{p+\gamma-1}{2}}\bigr\|_{L^1(\Omega)}^{2},
\end{aligned}
\end{equation}
where $K_1$ is a positive constant that will be fixed later. Remember that \eqref{stima_1} gives
\begin{align*}
     \frac{\dd}{\dt}\|\me\|_{L^p(\Omega)}^p  &  + \frac{k_Dp(p-1)}{(p+\gamma-1)^2}\bigl\|\nabla \me ^{\frac{p+\gamma-1}{2}}\bigr\|_{L^2(\Omega)}^2\\
     &\leq \frac{\chi^2 p(p-1)}{k_D}\into \me^{p-\gamma+1}|\nabla \ce|^2 \dx + \mu p\|\me\|_{L^p(\Omega)}^p - \mu p\|\me\|_{L^{p+1}(\Omega)}^{p+1}\\
     &\leq \frac{\chi^2 p(p-1)C_5^2}{k_D}\into \me^{p-\gamma+1} \dx + \mu p\|\me\|_{L^p(\Omega)}^p - \mu p\|\me\|_{L^{p+1}(\Omega)}^{p+1}.
\end{align*}
Using Young's inequality on the term involving $\me^{p-\gamma+1}$ and \eqref{est:Lp} we have
\begin{align*}
     \frac{\dd}{\dt}\|\me\|_{L^p(\Omega)}^p  &+\|\me\|_{L^p(\Omega)}^p   + \frac{k_D}{2}\bigl\|\nabla \me ^{\frac{p+\gamma-1}{2}}\bigr\|_{L^2(\Omega)}^2\\
     &\leq \frac{\chi^2 p^2 C_5^2}{k_D}\into \me^{p+\gamma-1} \dx +\left(\frac{\chi^2 p^2 C_5^2}{k_D}+\frac{(\mu+1)^{p+1}}{\mu^p}\left(\frac{p}{p+1}\right)^{p+1}\right)|\Omega|.
\end{align*}
Applying \eqref{est:pgammauno} we obtain
\begin{align*}
     \frac{\dd}{\dt}\|\me\|_{L^p(\Omega)}^p  &+\|\me\|_{L^p(\Omega)}^p   + \frac{k_D}{2}\bigl\|\nabla \me ^{\frac{p+\gamma-1}{2}}\bigr\|_{L^2(\Omega)}^2\\
     &\leq \frac{\chi^2 p^2 C_5^2}{k_D}\left[ \frac{nK_1}{n+2} \bigl\|\nabla\me^{\frac{p+\gamma-1}{2}}\bigr\|_{L^2(\Omega)}^{2}+\left(C_6^\frac{n+2}{2}\frac{2K_1^{-\frac{n}{2}}}{n+2}+C_7\right)\bigl\|\me^{\frac{p+\gamma-1}{2}}\bigr\|_{L^1(\Omega)}^{2}\right] \\
     &+\left(\frac{\chi^2 p^2 C_5^2}{k_D}+\frac{(\mu+1)^{p+1}}{\mu^p}\left(\frac{p}{p+1}\right)^{p+1}\right)|\Omega|.
\end{align*}
Setting
\begin{align*}
   & K_1 = \frac{k_D^2(n+2)}{2n\chi^2 p^2 C_5^2},
\end{align*}
we deduce the following estimate
\begin{equation}\label{est:stima_1_elab}
    \begin{aligned}
               \frac{\dd}{\dt}\|\me\|_{L^p(\Omega)}^p  +\|\me\|_{L^p(\Omega)}^p \leq & \frac{\chi^2 p^2 C_5^2}{k_D}\left[\left(\frac{2n\chi^2 p^2 C_5^2}{k_D^2(n+2)}\right)^{\frac{n}{2}}\frac{2C_6^\frac{n+2}{2}}{n+2}+C_7\right]\bigl\|\me^{\frac{p+\gamma-1}{2}}\bigr\|_{L^1(\Omega)}^{2} \\
     &+\left(\frac{\chi^2 p^2 C_5^2}{k_D}+\frac{(\mu+1)^{p+1}}{\mu^p}\left(\frac{p}{p+1}\right)^{p+1}\right)|\Omega|.  
    \end{aligned}
\end{equation}

Given $p_0\in\R$ as above, we then recursively define $p_k=2p_{k-1}-(\gamma-1)$, for $k\in \mathbb{N}$. Is it easy to check that
\[
\left(\frac{1}{2}+\frac{1}{2^k}\right)(\gamma-1)\leq \frac{p_k}{2^k}\leq p_0,
\]
for all $k\in \mathbb{N}$, that ensure an uniform in $k$ boundedness from above and below for the ratio $\frac{p_k}{2^k}$. Introducing
\[
 M_k \equiv \sup_{t\in(0,T)} \into \me ^{p_k}\dx,
\]
and integrating in time \eqref{est:stima_1_elab} with $p=p_k$ we obtain
\begin{equation}\label{est:stima_1_elab_2}
    \begin{aligned}
               M_k \leq & \|m_0\|_{L^{p_k}(\Omega)}^{p_k}+\frac{\chi^2 p_k^2 C_5^2}{k_D}\left[\left(\frac{2n\chi^2 p_k^2 C_5^2}{k_D^2(n+2)}\right)^{\frac{n}{2}}\frac{2C_6^\frac{n+2}{2}}{n+2}+C_7\right]M_{k-1}^2 \\
     &+\left(\frac{\chi^2 p_k^2 C_5^2}{k_D}+\frac{(\mu+1)^{p_k+1}}{\mu^p_k}\left(\frac{p_k}{p_k+1}\right)^{p_k+1}\right)|\Omega|\\
     \leq & p_k^{n+2}\frac{\chi^2 C_5^2}{k_D}\left(\frac{2n\chi^2  C_5^2}{k_D^2(n+2)}\right)^{\frac{n}{2}}\frac{2C_6^\frac{n+2}{2}}{n+2}M_{k-1}^2.  
    \end{aligned}
\end{equation}
For $k$ sufficiently large,  we can deduce the existence of a constant $C_8$ that depends on the upper bound for $\frac{p_k}{2^k}$ such that
\[
M_k \leq C_8^k M_{k-1}^2,
\]
that inductively leads to
\[
M_k \leq C_8^{k+\sum_{j=1}^{k-1}2^j(k-j)} M_{0}^{2^k}\leq C_8^{2}C_8^{2^{k}} M_{0}^{2^k}.
\]
Thanks to the lower bound of  $\frac{p_k}{2^k}$, by sending $k\to \infty$ we deduce the existence of $C_9>0$ such that
\[
\sup_{t\in(0,T)}\|\me(\cdot,t)\|_{L^\infty(\Omega)}\leq \limsup_{k\to\infty} M_k^{\frac{1}{p_k}}\leq \limsup_{k\to\infty}C_8^{2}C_8^{\frac{2^{k}}{p_k}} M_{0}^{\frac{2^k}{p_k}}\leq C_9.
\]

\end{proof}

The combination of all previous estimates provides the following global well-posedness result.

\begin{cor}[Global existence and boundedness of classical solutions to \eqref{eq:regulare}]\label{cor:classical}
Let $\mu, \chi>0$ and $\gamma$ under condition \eqref{cod_gamma}. Suppose that the initial condition $(\me_0,\ce_0, \de_0)$ satisfies \eqref{indataregul}. Then there exist a constant $C>0$ such that system \eqref{eq:regulare} has a classical solution $$(\me,\ce,\de)\in \left(C^0(\left[0,\infty\right)\times\bar{\Omega} )\cap C^{2,1}(\left(0,\infty\right)\times\bar{\Omega}) \right)^3$$ which exists globally in time and satisfies
\begin{equation*}
    \|\me(\cdot,t)\|_{L^\infty(\Omega)}+\|\ce(\cdot,t)\|_{W^{1,\infty}(\Omega)}+\|\de(\cdot,t)\|_{L^{\infty}(\Omega)}\leq C
\end{equation*}
for all $t\in \left(0,\infty\right)$.
\end{cor}

A direct integration on \eqref{stima_1}, together with the bounds in \eqref{stima_8}, gives the following bounds.
\begin{cor}\label{cor:nonlinear_estimates}
Let $\mu,\chi\geq 0$ and $\gamma$ under condition \eqref{cod_gamma}. Let $p\in (1,\infty)$. Then there exists a constant $C>0$ independent from $\ep$ such that
\begin{equation}\label{stima_9}
    \int_{0}^t\into D_\ep(\me)\me^{p-2}|\nabla\me|^2\dx \leq C(1+t),
\end{equation}
and
\begin{equation}\label{stima_10}
    \int_{0}^t\into |\nabla\me^{\frac{p+\gamma-1}{2}}|^2\dx \leq C(1+t),
\end{equation}
for all $t\in (0,\infty)$.
\end{cor}

\section{Convergence to global weak solutions} \label{sec:nonregularised}

In the present section we finally produce existence of global weak solution to system \eqref{eq:main_PM}.  The nonlinearity in the equation for $\me$ requires a stronger notion of convergence with respect to the one of $\ce$ and $\de$. For this purpose, we produce the following dual estimates for the time derivative of $\me$.

\begin{lem}\label{lem:timeder}
Fix $\chi,\mu>0$ and assume that $\gamma$ is under condition \eqref{cod_gamma}. Let $\theta>\max\left\{1,\frac{\gamma}{2}\right\}$. Then for $r>1$ and for any $T>0$ there exists a constant $C>0$ such that
\begin{equation}\label{stima_11}
    \bigl\|\partial_t \me^\theta\bigr\|_{L^1\left([0,T);(W_0^{1,r}(\Omega))^\ast\right)}\leq C,
\end{equation}
for any $\ep>0$.
\end{lem}
\begin{proof}
Let $\zeta\in C_0^\infty(\Omega)$ be such that $\|\zeta\|_{W_0^{1,r}(\Omega)}\leq 1$. We denote with $C_{\infty,f}$ the $\ep-$independent $L^\infty-$bound of a generic function $f$. An integration by parts gives
\begin{align*}
    \frac{1}{\theta}\into \partial_t \me^\theta \zeta \dx = & -(\theta-1)\into \me^{\theta-2} D_\ep (\me) |\nabla\me|^2 \zeta \dx -\into \me^{\theta-1} D_\ep (\me) \nabla\me\cdot\nabla \zeta \dx\\
    &+k_f\chi(\theta-1)\into \me^{\theta-1} \nabla\me\cdot \nabla\ce \zeta \dx+k_f\chi\into \me^{\theta} \nabla\ce\cdot \nabla\zeta \dx\\
    &+\into \me^{\theta-1} M(\me) \zeta \dx.
\end{align*}
For $T>0$, the density of $C_0^\infty(\Omega)$ in $W_0^{1,r}(\Omega)$ and an integration of the above equality in time, give
\begin{equation}\label{est:time0}
    \begin{aligned}
    \frac{1}{\theta} \bigl\|\partial_t \me^\theta\bigr\|_{L^1\left([0,T);(W_0^{1,r}(\Omega))^\ast\right)} \leq & (\theta-1)\int_0^T\into \me^{\theta-2} D_\ep (\me) |\nabla\me|^2 |\zeta| \dx\dt \\
    &+\int_0^T\into \me^{\theta-1} D_\ep (\me) |\nabla\me\cdot\nabla \zeta| \dx\dt\\
    &+k_f\chi(\theta-1)\int_0^T\into \me^{\theta-1} |\nabla\me\cdot \nabla\ce \zeta| \dx\dt\\
    &+k_f\chi\int_0^T\into \me^{\theta}| \nabla\ce\cdot \nabla\zeta| \dx\dt\\
    &+\int_0^T\into \me^{\theta-1} |M(\me) \zeta| \dx\dt.
\end{aligned}
\end{equation}

We now estimate each term on the r.h.s. separately. We start fixing $p>1$ such that $$\theta \geq \max\left\{p,\frac{p+\gamma-1}{2}\right\}.$$ Thanks to \eqref{stima_9} and the boundedness of $\zeta$ and $\me$ we can deduce that
\begin{equation}\label{est:time1}
\begin{aligned}
    \int_0^T\into \me^{\theta-2} D_\ep (\me) |\nabla\me|^2 |\zeta| \dx\dt \leq & C_{\infty,\zeta} C_{\infty,\me}^{\theta-p}  \int_0^T\into \me^{p-2} D_\ep (\me) |\nabla\me|^2  \dx\dt \\
    \leq & C_{\infty,\zeta} C_{\infty,\me}^{\theta-p} C(1+T)  . 
\end{aligned}
\end{equation}
The second term in \eqref{est:time0} can be similarly estimated by considering \eqref{stima_9} and the fact that $D_\ep$ as well $D$ is a continuous function. Indeed, by Young's inequality
\begin{equation}\label{est:time2}
\begin{aligned}
\int_0^T\into \me^{\theta-1} D_\ep (\me) |\nabla\me\cdot\nabla \zeta| \dx\dt \leq & \frac12C_{\infty,\nabla\zeta} \int_0^T\into \me^{p-2} D_\ep (\me) |\nabla\me|^2 \dx\dt \\
&+\frac12C_{\infty,\nabla\zeta} \int_0^T\into \me^{2\theta-p} D_\ep (\me) |\nabla\me|^2 \dx\dt \\
\leq &\frac12C_{\infty,\nabla\zeta} C(1+T) \\
&+\frac12C_{\infty,\nabla\zeta}C_{\infty,D} C_{\infty,\me}^{2\theta -p-\gamma+1} \int_0^T\into \me^{\gamma-1}  |\nabla\me|^2 \dx\dt \\
\leq &\frac12C_{\infty,\nabla\zeta} C(1+T) \\
&+\frac12C_{\infty,\nabla\zeta}C_{\infty,D} C_{\infty,\me}^{2\theta -p-\gamma+1} C(1+T),
\end{aligned}
\end{equation}
where in the last inequality we use \eqref{stima_10} with $p=2$. Using Young's inequality, the bound in \eqref{stima_10} and the bound for $\ce$ in \eqref{stima_8} we have
 \begin{equation}\label{est:time3}
\begin{aligned}
\int_0^T\into \me^{\theta-1} |\nabla\me\cdot \nabla\ce \zeta| \dx\dt \leq & \frac12 C_{\infty,\zeta}\int_0^T\into \me^{2\theta-2} |\nabla\me|^2 \dx\dt\\
&+\frac12 C_{\infty,\zeta}\int_0^T\into |\nabla\ce|^2 \dx\dt\\
\leq & \frac12 C_{\infty,\zeta}C_{\infty,\me}^{2\theta -p-\gamma-1}C(1+T)+\frac12 C_{\infty,\zeta} C_{\infty,\nabla\ce}^2 |\Omega|T. 
\end{aligned}
\end{equation}
Invoking again the bound for $\ce$ in \eqref{stima_8} we can easily bound
\begin{equation}\label{est:time4}
    \int_0^T\into \me^{\theta}| \nabla\ce\cdot \nabla\zeta| \dx\dt\leq C_{\infty,\nabla\zeta}C_{\infty,\me}^\theta  C_{\infty,\nabla\ce}^2|\Omega|T.
\end{equation}
Finally, by \eqref{ass_M} 
\begin{equation}\label{est:time5}
    \int_0^T\into \me^{\theta-1} |M(\me) \zeta| \dx\dt\leq \mu C_{\infty,\zeta}C_{\infty,\me}^\theta(1+C_{\infty,\me})|\Omega|T.
\end{equation}
Putting together \eqref{est:time1}-\eqref{est:time5}, we can deduce the existence of a constant $C>0$ independent on $\ep$ such that \eqref{est:time0} reduce to
\begin{equation*}
  \bigl\|\partial_t \me^\theta\bigr\|_{L^1\left([0,T);(W_0^{1,r}(\Omega))^\ast\right)} \leq 
    C(1+T).
\end{equation*}
\end{proof}

\begin{lem}\label{lem:conv}
Let $\mu,\chi\geq 0$ and $\gamma$  under condition \eqref{cod_gamma}. For any $T\in (0,\infty]$ there exist $c$ and $d$ belonging to $L^{2}_{loc}([0,T);W^{1,2}(\Omega))\cap L^{\infty}\left([0,T)\times\Omega\right)$ and  $m\in L^{\infty}\left([0,T)\times\Omega\right)$ such that, up to a non relabel sub-sequence,
\begin{align}
     \ce \overset{\ast}{\rightharpoonup} c,\, \de \overset{\ast}{\rightharpoonup} d  &\quad\mbox{in } L^{\infty}\left((0,T)\times\Omega \right),\label{conv2}\\
     \nabla\ce\overset{\ast}{\rightharpoonup} \nabla c,\,  &\quad \mbox{in } L^{\infty}((0,T)\times\Omega),\label{conv3}\\
     \me \rightarrow m &\quad \mbox{a.e. in }(0,T)\times\Omega,\label{conv4}\\
      \me \rightarrow m &\quad \mbox{in }L^2([0,T);L^2(\Omega)),\label{conv5}\\
      D_\ep(\me)\nabla \me \rightharpoonup D(m)\nabla m &\quad \mbox{in }L_{loc}^2([0,T);L^2(\Omega)),\label{conv6}
\end{align}
as $\ep \to 0$.
\end{lem}
\begin{proof}
The desired convergences for $\ce$ and $\de$ in \eqref{conv2} are direct consequences of the $L^\infty-$estimates in \eqref{stima_8}, combined with \eqref{stima_sign} and \eqref{stima_10}. For $T>0$ and any $p>1$ the bounds for $\me$ in \eqref{stima_8} and \eqref{stima_10} give that the sequence $\left\{\me^{\frac{p+\gamma-1}{2}}\right\}_\ep$ for $\ep\in(0,1)$ is bounded in $L^2\left([0,T);W^{1,2}(\Omega)\right)$. In addition, Lemma \ref{lem:timeder} gives that $\left\{\partial_t\me^{\frac{p+\gamma-1}{2}}\right\}_\ep$ for $\ep\in(0,1)$ is bounded in $L^1\left([0,T);W^{1,2}(\Omega)^\ast\right)$. Thus, we are in the position of apply Aubin-Lions lemma \cite[Corollary 4]{Simon86} in order to deduce the relative compactness  of the sequence $\left\{\me^{\frac{p+\gamma-1}{2}}\right\}_{\ep\in(0,1)}$ in $L^2\left([0,T);L^2(\Omega)\right)$ and the existence in the space $L^2\left([0,T);L^2(\Omega)\right)$ of a function that we can directly identify with $m^{\frac{p+\gamma-1}{2}}$ such that, along a  non relabel sub-sequence, $\me^{\frac{p+\gamma-1}{2}}\to m^{\frac{p+\gamma-1}{2}}$ in $L^2\left([0,T);L^2(\Omega)\right)$ as $\ep\to 0$. The uniform $L^\infty-$boundedness of $\me$ and Lebesgue's dominated convergence Theorem allow to deduce \eqref{conv4} and \eqref{conv5}. 

We are left in proving \eqref{conv6}. Invoking again the uniform bound for $\me$ in \eqref{stima_8}, the regularity of $D_\ep$ and \eqref{stima_9} setting $p=2$, we have that $D_\ep(\me)\nabla \me$ is uniformly bounded in $L_{loc}^2\left([0,T);L^2(\Omega)\right)$, thus, up to sub-sequence, it converges weakly in $L^2\left(0,T;L^2(\Omega)\right)$. Notice that the weak convergence of $D_\ep(\me)\nabla \me$ implies the weak convergence of $\nabla \Phi_e(\me)$ to some function $F$, where $\Phi_\ep$ is defined in \eqref{phiepsilon}. In order to identify $F=D(m)\nabla m$ we notice that the convergence in \eqref{conv4} ensure that
$\Phi_\ep (\me+\ep) \to \Phi (m)$ as $\ep \to 0$ a.e. in $\Omega\times (0,T)$, with $\Phi$ as in \ref{ass_D}. The strong convergence and Lebesgue's dominated convergence Theorem allow to identify the weak limit $F$ with $\nabla \Phi(m)=D(m)\nabla m$, concluding the proof.
\end{proof}

We are now in the position of proving our main result. 
\begin{proof}[Proof of Theorem \ref{thm1}] Let $T\in (0,\infty]$ and let $(\me,\ce,\de)$ be a global and bounded classical solution to \eqref{eq:regulare} in the sense of Corollary \ref{cor:classical} on $[0,T)$.
Let $\phi\in C_0^{\infty}\left(\overline{\Omega}\times \left[0,T\right)\right)$ be a test function. Multiplying each equation of \eqref{eq:regulare} by $\phi$ and integrating by parts we get
\begin{equation*}
    \begin{aligned}
-\int_0^T\into \me\phi_t\, \dx\dt-\int_{\Omega} m_{\ep,0}(x,0)\phi(x,0)\,\dx=&-\int_0^T\into D_\ep(\me)\nabla \me\nabla\phi \,\dx\dt\\
&+\chi\int_0^T\into f(\me)\nabla \ce\nabla\phi \,\dx\dt+\int_0^T\into M(\me)\phi \,\dx\dt,
\end{aligned}
\end{equation*}
\begin{equation*}
  -\int_0^T\into \ce\phi_t\,\dx\dt-\int_{\Omega}c_{\ep,0}(x,0)\phi(x,0)\,\dx=-\int_0^T\into\nabla \ce\nabla\phi\,\dx\dt+\int_0^T\into(\me+\de-\ce)\phi\,\dx\dt,
\end{equation*}
\begin{equation*}
    -\int_0^T\into \de\phi_t\,\dx\dt-\int_{\Omega}d_{\ep,0}(x,0)\phi(x,0)\dx=\int_0^T\into  h(\me)(1-\de)\phi\,\dx\dt.
\end{equation*}
The convergences obtained in \eqref{conv2}-\eqref{conv6}, together with \eqref{indataregul}, allow to pass to the limit in all the terms in the above equalities, in particular in the nonlinear terms involving $f$ and $h$. 
Then we can conclude that the limiting triple $(m,c,d)$ obtained in Lemma \ref{lem:conv} is a weak solution to \eqref{eq:main_Gen} in the sense of Definition \ref{def:weak_sol}.  Finally, the boundedness is a direct consequence of \eqref{conv2}, \eqref{conv3} and \eqref{conv4}.
\end{proof}

\begin{figure}[H]
  \centering
        \includegraphics[width=7cm,height=6cm]{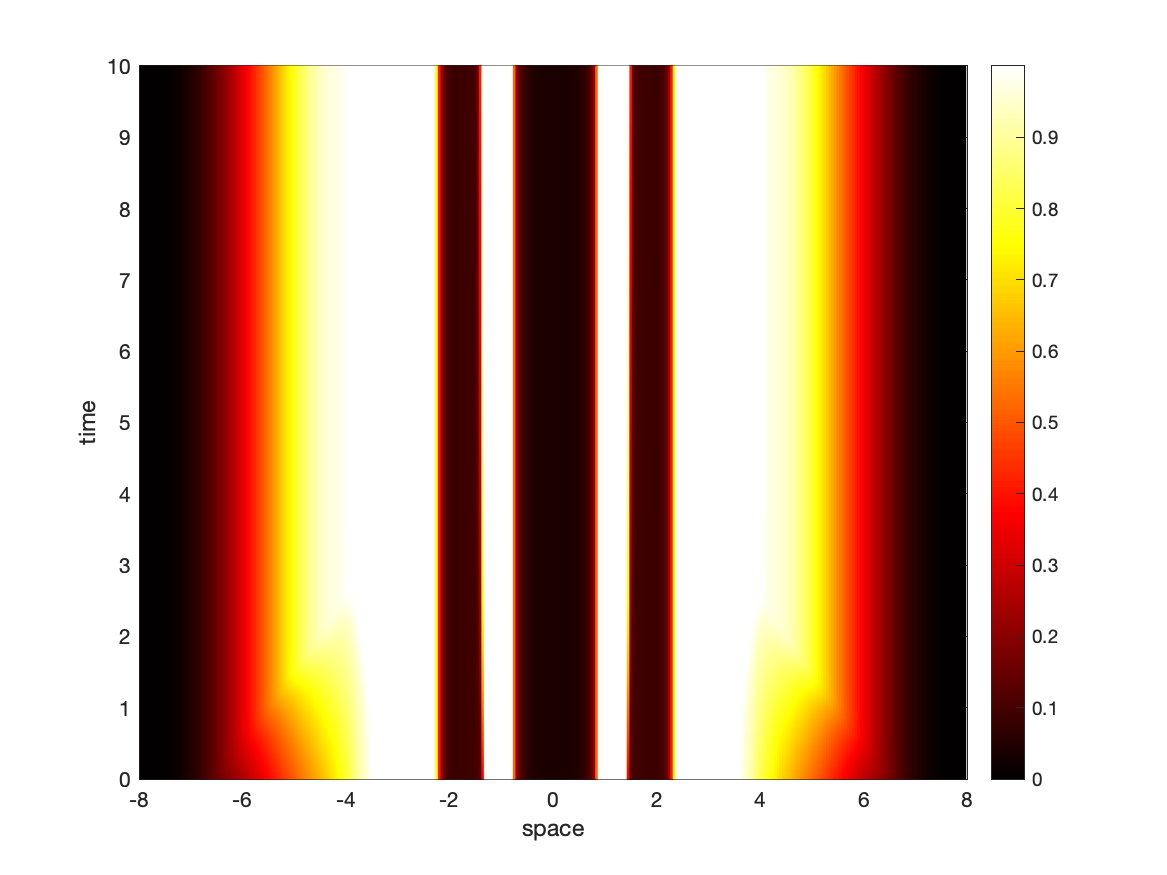}
    \includegraphics[width=7cm,height=6cm]{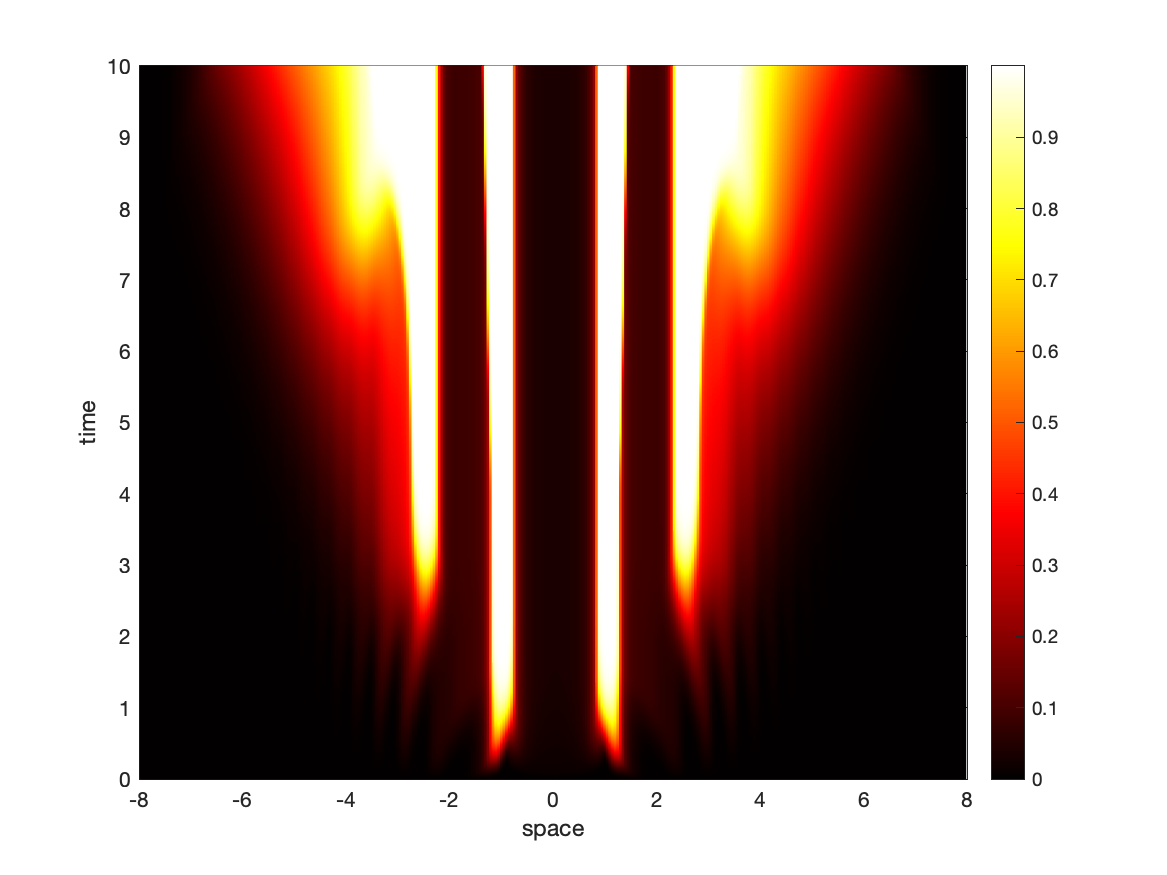}\\
       \includegraphics[width=7cm,height=6cm]{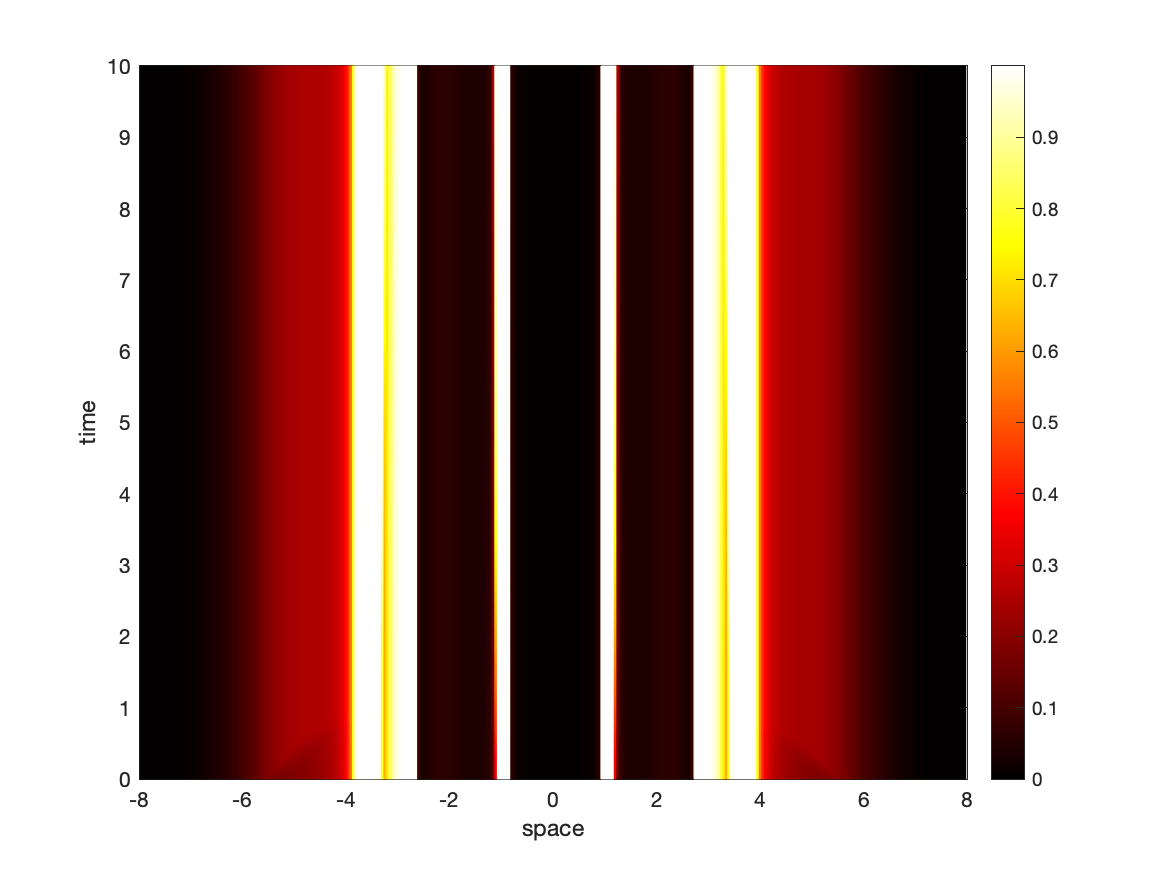}
 \includegraphics[width=7cm,height=6cm]{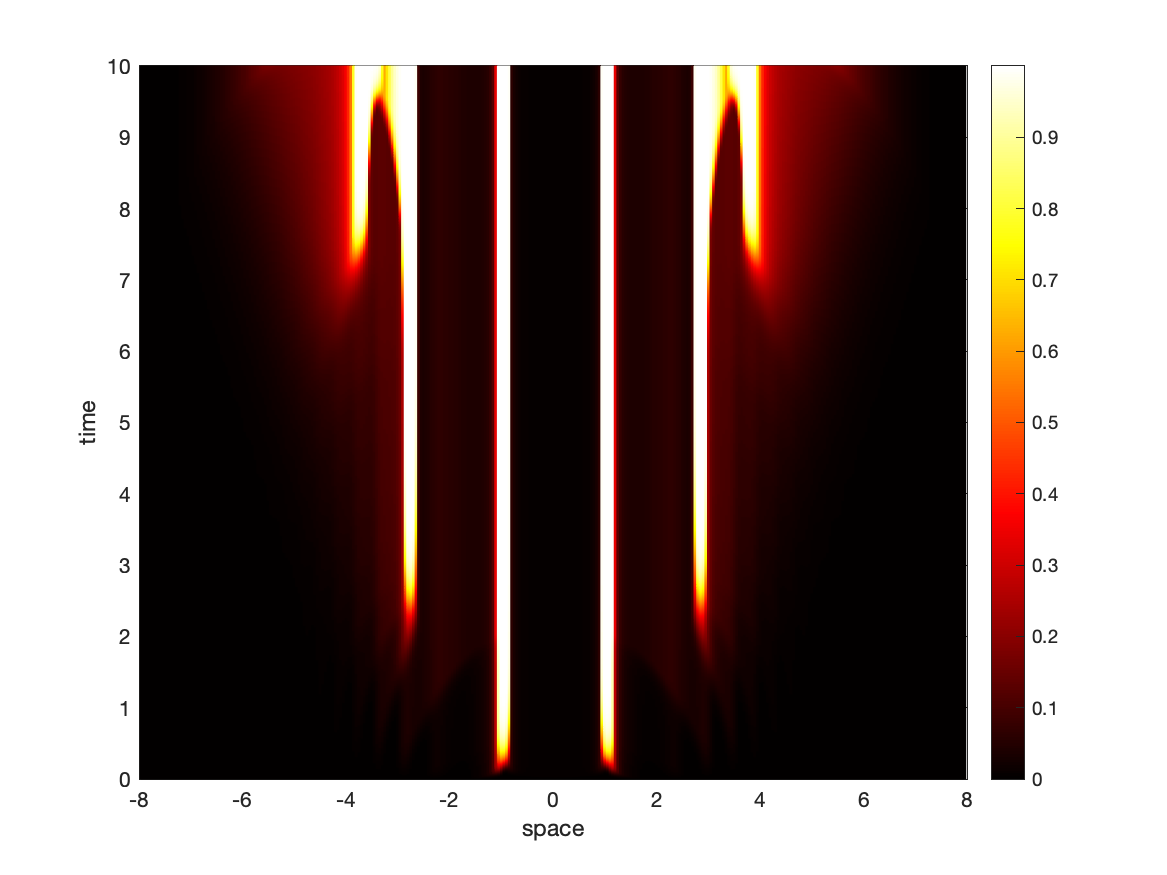}
\caption{Evolution in time of damage oligodendrocytes under the action of linear diffusion (left column) and quadratic $(\gamma=2)$ porous medium diffusion (right column) for activated macrophages in dimension $d=1$. In the top row we used a chemotaxis coefficient $\chi=4$ while in the bottom row we used a chemotaxis coefficient $\chi=10$.}
\label{fig:comp_1}
\end{figure}

\begin{figure}[H]
  \centering
        \includegraphics[width=7cm,height=6cm]{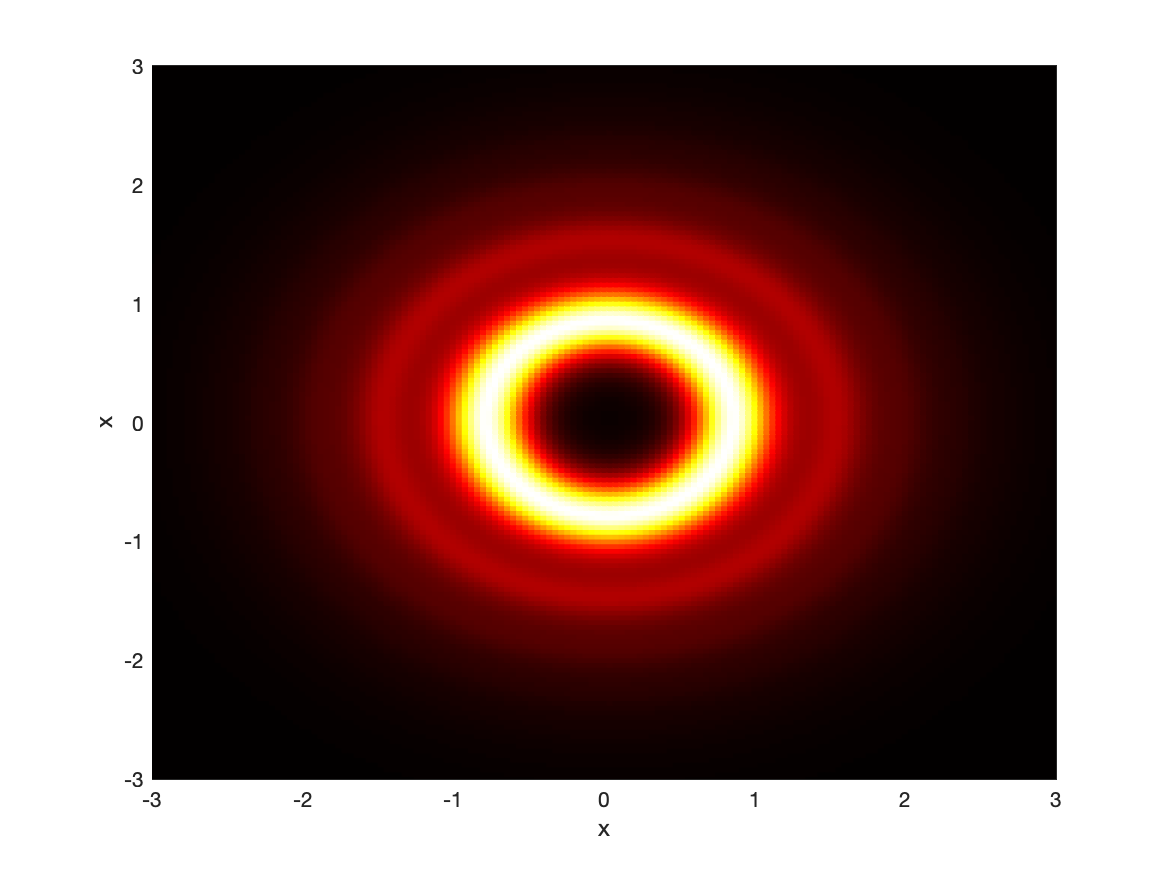}
    \includegraphics[width=7cm,height=6cm]{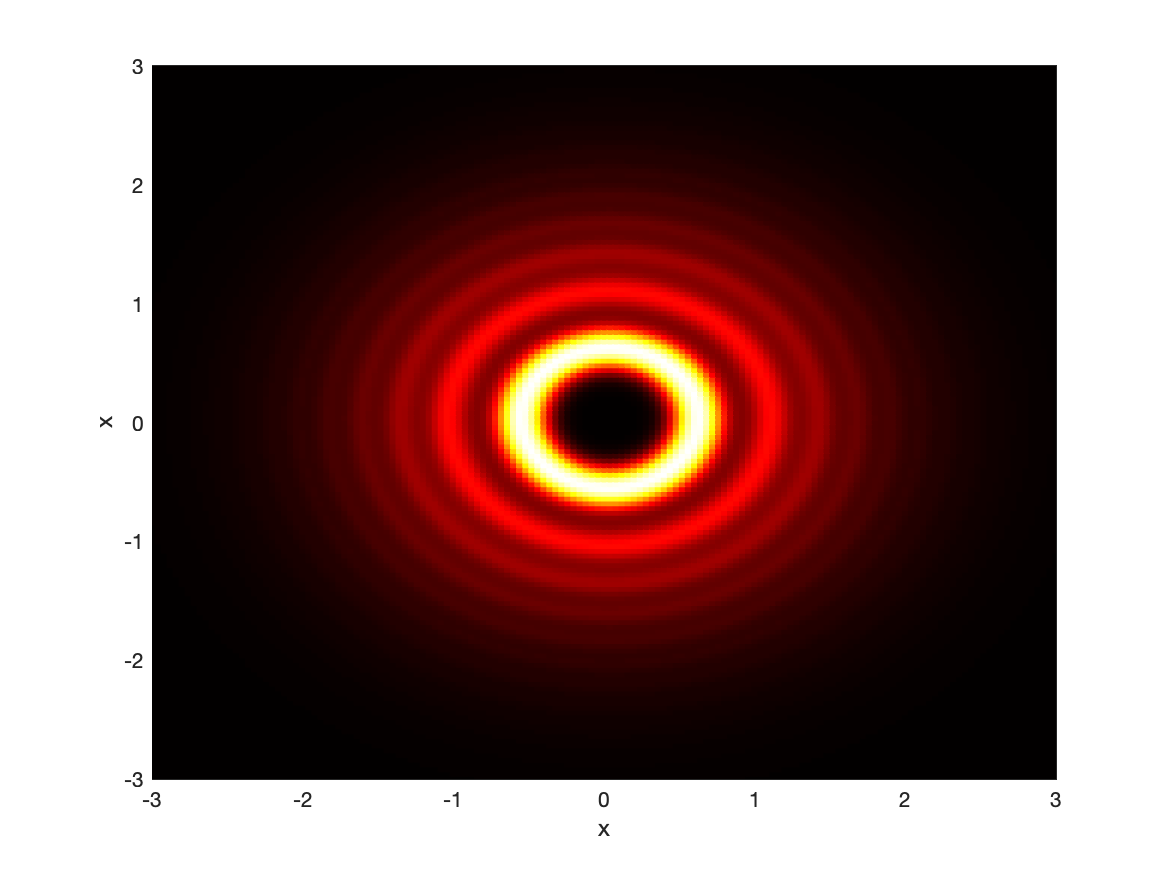}\\
       \includegraphics[width=7cm,height=6cm]{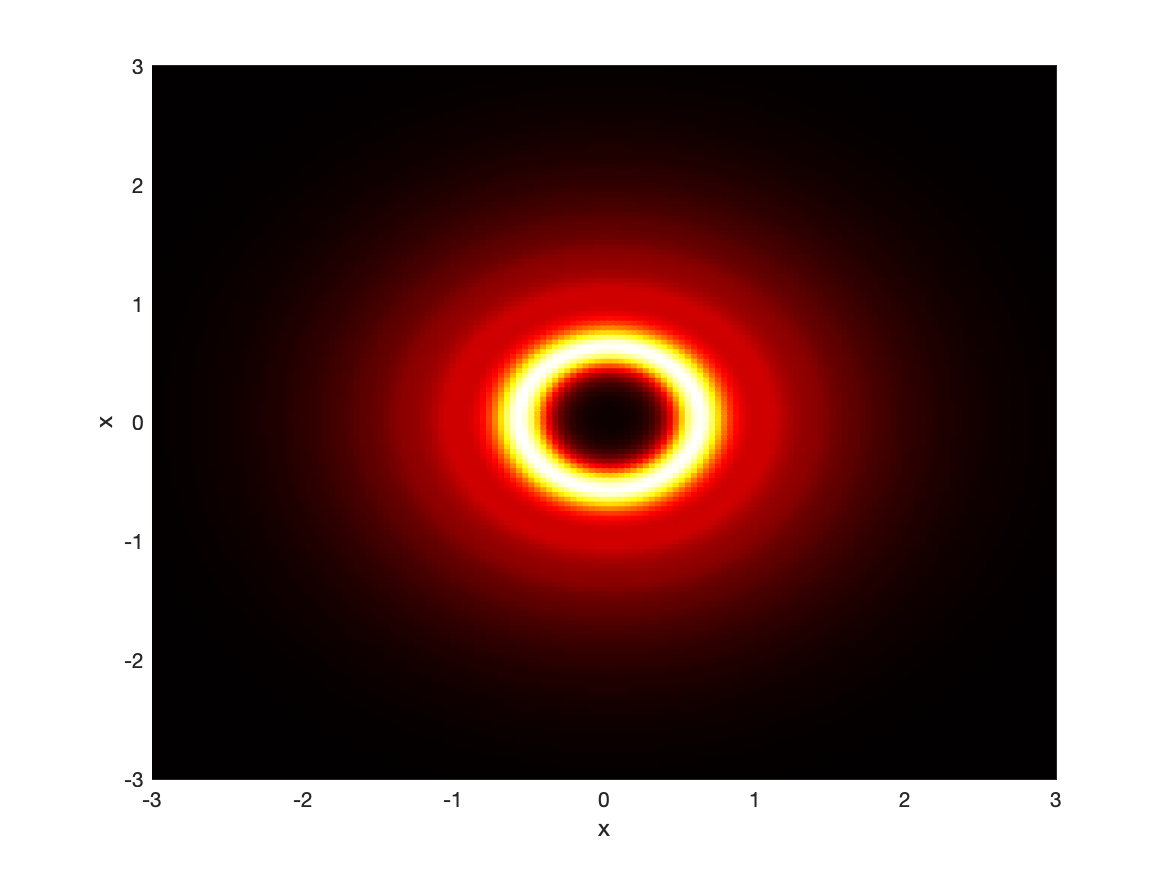}
 \includegraphics[width=7cm,height=6cm]{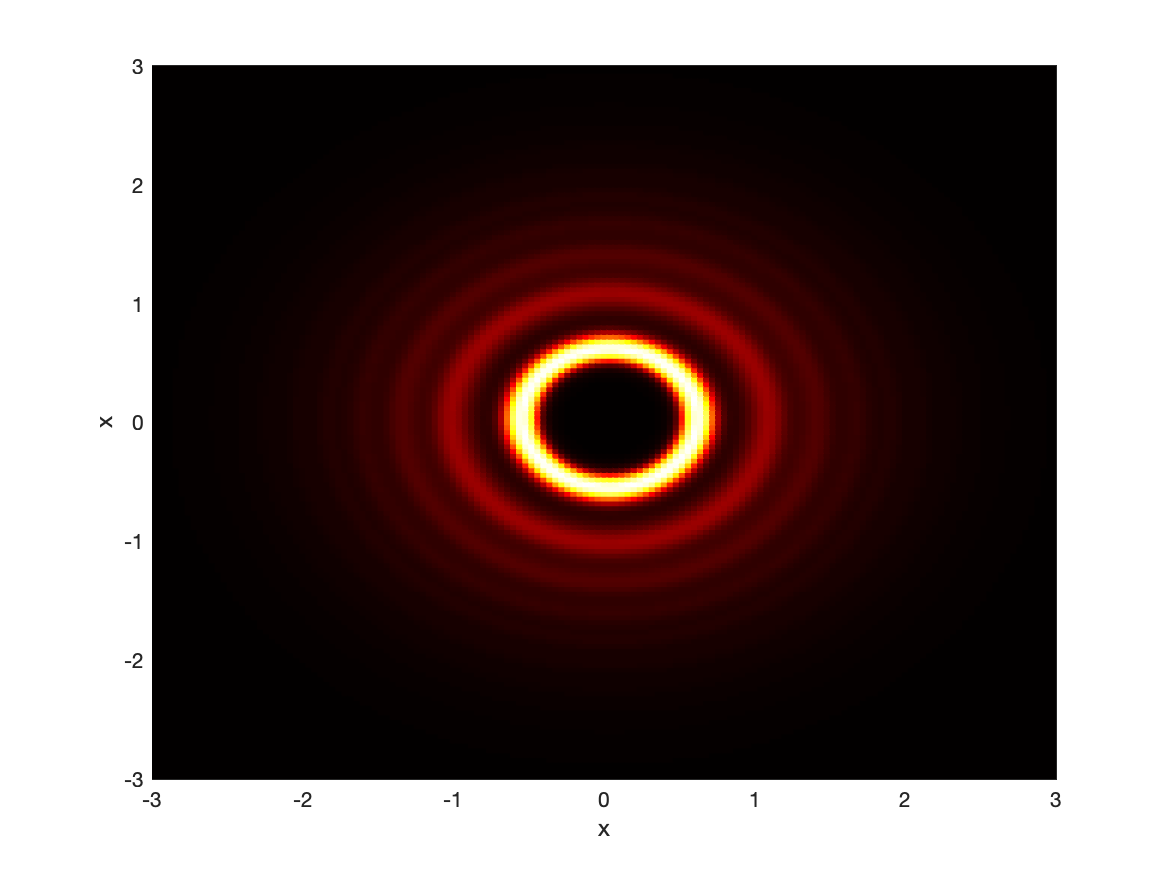}
\caption{Evolution in time of damage oligodendrocytes under the action of linear diffusion (left column) and quadratic $(\gamma=2)$ porous medium diffusion (right column) for activated macrophages in dimension $d=2$. In the top row we consider a chemotaxis coefficient $\chi=4$ while in the bottom row we consider a chemotaxis coefficient $\chi=10$. }
\label{fig:comp_2}
\end{figure}

\section{Conclusion and perspectives}\label{sec:conclusion}

We conducted an investigation into the existence of global weak solutions for a system of chemotaxis type with nonlinear degenerate diffusion, which arises in the modeling of Multiple Sclerosis disease. Our approach involved employing a regularization procedure on the degenerate diffusion coefficient, enabling us to obtain appropriate a priori estimates on the regularized problem. It is worth noting that the system we studied is a modification of a model originally designed with linear diffusion.

Intriguingly, as highlighted in Figures \ref{fig:comp_1} and \ref{fig:comp_2}, the incorporation of nonlinear diffusion leads to the formation of more interesting patterns. These patterns can potentially represent various types of plaque formation observed in Multiple Sclerosis and deserve further studies. Numerical simulations are performed according to the scheme sketched in Appendix \ref{sec:appendix num}

Moving forward, an interesting extension of our current work would involve exploring possible stationary states and radially symmetric solutions. By investigating these aspects, we can gain further insights into the long-term behavior and spatial characteristics of the system under consideration.

\section*{Acknowledgments}

The research of SF is  supported by the Ministry of University and Research (MIUR), Italy under the grant PRIN 2020- Project N. 20204NT8W4, Nonlinear Evolutions PDEs, fluid dynamics and transport equations: theoretical foundations and applications and  by the  INdAM project N. E55F22000270001 ``Fenomeni di trasporto in leggi di conservazione e loro applicazioni''. SF is also supported by University of L'Aquila 2021 project 04ATE2021 - ``Mathematical Models For Social Innovations: Vehicular And Pedestrian Traffic, Opinion Formation And Seismology.'' ER and SF are supported by the INdAM project N.E53C22001930001 ``MMEAN-FIELDSS''.

\begin{appendices}
\section{Collection of useful inequalities}\label{pre}
In this Appendix we list, without proofs, several helpful auxiliary results that are often used in the analysis of the systems similar to \eqref{eq:main_Gen}. Most of the can be found in \cite{li2016} and references cited there.
Firstly, we recall the Gagliardo-Nirenberg interpolation inequality \cite{niremberg} in the form \cite{li2016,winkler2002}
\begin{lem}[\cite{li2016}, Lemma 2.3]\label{Gag_Nir}
Let $\Omega\subset\mathbb{R}^n$ be a bounded domain with smooth boundary $\partial\Omega$. Let $R  \geq 1$, $0<Q\leq P\leq \infty$, $S>0$ be such that
\begin{equation}\label{condition_Gag_Nir}
    \frac{1}{R}\leq \frac{1}{n}+\frac{1}{P}.
\end{equation}
Then there exists a positive constant $C>0$ such that, 
\begin{equation}\label{eq:Gag_Nir}
\|u\|_{L^P(\Omega)}\leq C\left(\|\nabla u\|^{A}_{L^R(\Omega)}\|u\|_{L^{Q}(\Omega)}^{1-A}+\|u\|_{L^{S}(\Omega)}\right),
\end{equation}
for all $u\in W^{1,R}\cap L^{Q}(\Omega)$, where
\[
A=\displaystyle{\frac{\frac{1}{Q}-\frac{1}{P}}{\frac{1}{Q}+\frac{1}{n}-\frac{1}{R}}}.
\]
\end{lem}

The next Lemma ensures that the conditions in Lemmas \ref{lemm:primo_incubo} and \ref{lemm:secondo_incubo} below are simultaneously satisfied. 
\begin{lem}[\cite{li2016}, Lemma 3.8]\label{exponents}
Let $n\geq 2$ and $\gamma$ be under condition \eqref{cod_gamma}. Then there exists unbounded sequences $\left\{p_k\right\}_{k\in\N}$ and $\left\{q_k\right\}_{k\in\N}$ such that for each $k\in\N$ $q_k, p_k>1$ and
\begin{align*}
    &  p_k>\max\left\{\gamma-\frac{n-2}{nq_k},\frac{2q_k(n-2)}{2q_k+n-2}\right\},\\
    &   \frac{q_k(p_k-\gamma+1)}{q_k-1} \left(\frac{n-\frac{nq_k-n+2}{q_k(p_k-\gamma+1)}}{(p_k+\gamma-1)n+2-n}\right)<1,\\
    &\frac{1}{q_k}>\frac{2}{p_k+\gamma-1}\left(\frac{\frac{p_k+\gamma-1}{2}-\frac{(p_k+\gamma-1)(2q_k+n-2)}{4nq_k}}{\frac{p_k+\gamma-1}{2}+\frac{1}{n}-\frac{1}{2}}\right).
\end{align*}
\end{lem}

The following Poincaré-type inequality is useful for our purpose. A detailed proof can be found in \cite[Lemma 2.2]{li2016}.
\begin{lem}\label{lemm:poin}
Let $\Omega\subset \R^n$ be a bounded smooth domain. Let $\alpha>0$ and $p\in(1,+\infty)$. Then there exists $C > 0$ such that
\[
\|u\|_{W^{1,p}(\Omega)}\leq C \left(\|\nabla u\|_{L^p(\Omega)}+\left(\into |u|^\alpha \dx\right)^{\frac{1}{\alpha}}\right),
\]
for all $u\in W^{1,p}(\Omega)$.
\end{lem}

The following lemma furnishes a proper estimate of a boundary integral and was proven in \cite[Lemma 2.1]{li2016} using previous results presented in \cite[Proposition 3.2]{ishida2014}.

\begin{lem}\label{lemm:boundary}
Let $\Omega\subset \R^n$ be a bounded domain with smooth boundary. Let $q\in\left[1,+\infty\right)$ and $M\geq 0$. Then, for any $\eta>0$ there is $C_\eta>0$ such that for any $u\in C^2(\bar{\Omega})$ with 
\[\frac{\partial u}{\partial \nu}=0,\, \mbox{ on }\, \partial \Omega,\, \mbox{ and }\, \into |\nabla u| \dx\leq M,\]
the inequality
\[
\int_{\partial\Omega}|\nabla u|^{2q-2}\frac{\partial |\nabla u|^{2}}{\partial \nu} \dd\sigma \leq \eta\into |\nabla|\nabla u|^q|^2 \dx +C_\eta,
\]
holds.
\end{lem}

Further, we will make use of the notion of Neumann-heat semigroup. Let $\Omega$ be an arbitrary domain in $\R^n$ and let $-\Delta$ denote the Neumann-Laplacian in $L^2(\Omega)$, that is the Laplacian on $L^2(\Omega)$ subject to Neumann-boundary conditions (see \cite{Davies,Winkler2010JDE} for its precise definition and properties). Then $-\Delta$ is a nonnegative self-adjoint operator and it generates a $C^0-$semigroup $e^{-t\Delta}$ on $L^2(\Omega)$. Note that $u = e^{-t\Delta}f$ solves the heat equation $u_t -\Delta u = 0$ in $\Omega\times(0,+\infty)$, $u\in C(\Omega\times(0,+\infty))$ and 
\[
\frac{\partial u}{\partial \nu} = 0\,\mbox{ on }\, \partial\Omega\times(0,+\infty).
\]
There exists a positive $C^\infty-$function $G : \Omega\times\Omega\times(0,+\infty)\to\R$ such that $$e^{-t\Delta}f(x) = \into G(x,y,t)f(y)\dd y$$ for any $f \in L^p(\Omega)$, $1\leq p\leq +\infty$. Moreover we recall standard regularity results that are relevant for the second equation of system \eqref{eq:main_Gen}.
\begin{prop}[\cite{ke}]\label{Prop2}
Suppose $k\in(1,+\infty)$ and $g\in L^k((0,T);L^k(\Omega))$. Consider the following evolution problem:
\begin{equation}
\begin{cases}
c_t-\Delta c+c=g, \ \ (x,t)\in\Omega\times(0,T), \\
\displaystyle{\frac{\partial c}{\partial\nu}}=0, \ \ (x,t)\in\partial\Omega\times(0,T),\\
c(x,0)=c_0(x), \ \ (x,t)\in\Omega.
\end{cases}
\end{equation}
For each $c_0\in W^{2,k}(\Omega)$ such that $\displaystyle{\frac{\partial c_0}{\partial \nu}}=0$ and any $g\in L^k((0,T); L^k(\Omega))$, there exists a unique solution $c\in W^{1,k}((0,T); L^k(\Omega))\cap L^k((0,T); W^{2,k}(\Omega))$. 
\end{prop} 

\section{Local existence of weak solutions for the regularised system}\label{sec:appendix united}


This appendix is devoted in proving the existence of local classical solutions to the regularised system \eqref{eq:regulare}, as stated in Proposition \ref{locexistregulare}.
\begin{proof}[Proof of Proposition \ref{locexistregulare}]
Fix $\epsilon\in (0,1)$ and let $T>0$ be a small final time to be fixed later. Let $K>0$ be such that $\|\me_0\|_{L^\infty(\Omega)}\leq K$. Consider the space
\[
 X = \left\{ f\in L^{\infty}(\Omega\times (0,T))\,\bigr| \,f\leq K+1\, \mbox{ a.e. in }\, \Omega\times (0,T) \right\}.
\]
Fix $\bar{m}\in X$ and consider the following truncated functions
\begin{equation}\label{eq:D_hat}
    \hat{D}_\epsilon(z)=\begin{cases}
    D_\epsilon(0) & z<0,\\
    D_\epsilon(z) & 0\leq z\leq K+1,\\
    D_\epsilon(K+1) & z> K+1,
    \end{cases}
\end{equation}
and
\begin{equation}\label{eq:m_hat}
    \hat{m}(x,t)=\begin{cases}
    \bar{m}(x,t) & \mbox{if }\,x\in\Omega\, \mbox{ and }\,t\in(0,T),\\
   0 &\mbox{if }\,x\in\Omega\, \mbox{ and }\,t\geq T,
    \end{cases}
\end{equation}
Standard ordinary differential equations theory ensures that the following equation
\begin{equation}\label{eq:froz_d}
\begin{cases}
\partial_t d  =- h(\hat{m})d +  h(\hat{m}), & x\in\Omega,\,t>0\\
 d(x,0)=\de_0(x), & x\in \Omega.
\end{cases}
\end{equation}
admits a globally defined solution $d$ that is bounded since $h(\hat{m})$ is bounded. By introducing the additional truncated function
\[
 \hat{d}(x,t)=\begin{cases}
    d(x,t) & \mbox{if }\,x\in\Omega\, \mbox{ and }\,t\in(0,T),\\
   0 &\mbox{if }\,x\in\Omega\, \mbox{ and }\,t\geq T,
    \end{cases}
\]
invoking the standard existence theory for linear parabolic problems, e.g. \cite{ladyzhenskaia1988linear}, we can argue that the problem
\begin{equation}\label{eq:froz_c}
\begin{cases}
 \partial_t c=\Delta c-c+\hat{m}+\hat{d} & x\in\Omega,\,t>0,\\
 \frac{\partial c}{\partial \nu} = 0 & x\in\partial\Omega,\,t>0,\\
 c(x,0)=\ce_0(x) & x\in \Omega,
\end{cases}
\end{equation}
admits a globally defined weak solution. Moreover, for any $q>\frac{d+2}{2}$ there exists $k_1>0$ such that
\[
\|\nabla c\|_{L^q\left(\Omega\times(0,T)\right)}\leq k_1.
\]
Given $c$ a solution to \eqref{eq:froz_c}, we can consider the equation
\begin{equation}\label{eq:froz_m}
\begin{cases}
 \partial_t m=\nabla\left(\hat{D}_\epsilon(\bar{m})\nabla m\right) -\chi \nabla\left(h(\bar{m})\nabla c\right)+M(\bar{m}) & x\in\Omega,\,t>0,\\
 \frac{\partial m}{\partial \nu} = 0 & x\in\partial\Omega,\,t>0,\\
 m(x,0)=\me_0(x) & x\in \Omega.
\end{cases}
\end{equation}
Since $h(\bar{m})\nabla c\in L^q\left(\Omega\times(0,1)\right) $, invoking again standard parabolic theory, we may deduce that there exist $k_2,k_3>0$ such that
\[
 \|m\|_{L^\infty\left(\Omega\times(0,1)\right)}\leq k_2\, \mbox{ and }\,\|m\|_{C^{\alpha,\frac{\alpha}{2}}(\bar{\Omega}\times \left[0,1\right)}\leq k_3.
\]
The combination of the two bounds above allow to control the $ L^\infty-$norm of $m$ as follows
\[
 \|m\|_{L^\infty\left(\Omega\times(0,1)\right)}\leq  \|\me_0\|_{L^\infty\left(\Omega\right)}+k_3t^{\frac{\alpha}{2}},\,\mbox{ for all }t\in(0,1).
\]
By choosing $T\in(0,1)$ such that $k_3 T^{\frac{\alpha}{2}}\leq 1$ we have that  $\|m\|_{L^\infty\left(\Omega\times(0,1)\right)}\leq K+1$, thus we can easily deduce that the solution operator associated to \eqref{eq:froz_m}, $S(\bar{m})=m$ maps $X$ into itself, is continuous and $\overline{S(X)}$ is compact in $L^\infty\left(\Omega\times(0,T)\right)$. Then,  Schauder fixed point theorem provide the existence of a fixed point $m$ for $S$ in $X$. Since $m\in X$ we have that actually $\hat{D}_\epsilon(m)\equiv D_\epsilon(m)$, and recalling the definitions for the truncated functions $\hat{m}$ and $\hat{d}$ we can deduce that the triple $(m,c,d)$ actually solves \eqref{eq:regulare} on $\Omega\times(0,T)$ classically. This triple depends on the fixed parameter $\epsilon$ and the nonnegativy can be deduced from parabolic comparison principle. Finally, \eqref{localtoinf} is a consequence of standard extensibility argument, since $T$ only depends on the initial data.
\end{proof}
\section{Proofs of technical Lemmas}\label{Proofs}

We collect in this Appendix a list of proofs that are rather technical but standard in the literature.  
\begin{proof}[Proof of Lemma \ref{lem:nabla_c}]
    We first recall the following identities
\begin{equation}\label{ide_1}
    2\nabla f\cdot\nabla \Delta f = \Delta|\nabla f|^2 - 2|D^2f|^2,
\end{equation}
and
\begin{equation}\label{ide_2}
    \nabla |\nabla f|^{2q-2} = (q-1)|\nabla f|^{2q-4}\nabla |\nabla f|^2 ,
\end{equation}
that hold for all smooth functions $f$. From \eqref{eq:regulare_c}, a direct computation shows that
\begin{align*}
     \frac{1}{q}\frac{\dd}{\dt}\into |\nabla\ce|^{2q}  =& 2\into |\nabla \ce|^{2q-2}\nabla \ce \cdot \nabla\Delta \ce\dx -2\into |\nabla \ce|^{2q-2}\nabla \ce \cdot\nabla\ce\dx\\
     & +2\into |\nabla \ce|^{2q-2}\nabla \ce \cdot\nabla (\de +\me)\dx.
\end{align*}
Applying identity \eqref{ide_1} and integrating by parts we have
\begin{align*}
     2\into |\nabla \ce|^{2q-2}\nabla \ce \cdot \nabla\Delta \ce\dx  = &\into |\nabla \ce|^{2q-2}\left(\Delta|\nabla \ce|^2 - 2|D^2\ce|^2\right)\dx \\
     =& -(q-1)\into |\nabla \ce|^{2q-4}|\nabla|\nabla \ce|^2|^2 \dx + \int_{\partial\Omega}|\nabla \ce|^{2q-2}\frac{\partial |\nabla \ce|^{2}}{\partial \nu} \dd\sigma\\
    &-2\into |\nabla \ce|^{2q-2} |D^2\ce|^2\dx,
\end{align*}
while an integration by parts, \eqref{ide_2} and Young's inequality give
\begin{align*}
2\into |\nabla \ce|^{2q-2}\nabla \ce \cdot\nabla (\de +\me)\dx  = &  -2(q-1)\into (\de+\me )|\nabla \ce|^{2q-4}\nabla|\nabla \ce|^2\cdot \nabla \ce \dx \\
&-2 \into (\de+\me )|\nabla \ce|^{2q-2} \Delta \ce \dx\\
 \leq &\frac{q-1}{2}\into |\nabla \ce|^{2q-4}|\nabla|\nabla \ce|^2|^2 \dx\\
 &+2(q-1)\into (\de+\me )^2|\nabla \ce|^{2q-2}\dx\\
 & + \frac{2}{n}\into |\nabla \ce|^{2q-2}n|D^2\ce|^2\dx\\
 &+\frac{n}{2}\into (\de+\me )^2|\nabla \ce|^{2q-2}\dx,
\end{align*}
where we used $|\Delta c_\ep|^2\leq n|D^2 c_\ep|^2$. Adding together the two estimates above we obtain
\begin{align*}
      \frac{1}{q}\frac{\dd}{\dt}\into |\nabla\ce|^{2q}\dx+2\into |\nabla\ce|^{2q}\dx 
     \leq & -\frac{q-1}{2}\into |\nabla \ce|^{2q-4}|\nabla|\nabla \ce|^2|^2\dx\\
     & +\left(2(q-1)+\frac{n}{2}\right)\into (\de+\me )^2|\nabla \ce|^{2q-2}\dx\\
     &+ \int_{\partial\Omega}|\nabla \ce|^{2q-2}\frac{\partial |\nabla \ce|^{2}}{\partial \nu} \dd\sigma.
\end{align*}
Estimate \eqref{stima_2} can be deduced once we properly estimate the boundary integral.  Note that, by denoting with $e^{\Delta t}$ the Neumann-heat semigroup, see \cite{Davies,QuiSou}, and  using the  Duhamel formula we can estimate
\begin{align*}
    \|\nabla\ce(.,t)\|_{L^1(\Omega)}\leq& \|\nabla e^{\Delta (t-1)} \ce_0\|_{L^1(\Omega)}+\int_{0}^t\|\nabla e^{\Delta (t-\tau)}(\me(\cdot,\tau)+\de(\cdot,\tau))\|_{L^1(\Omega)}\dd\tau.
\end{align*}
Calling $\lambda_1>0$ the first positive eigenvalue of $-\Delta$ and  using the $L^p-L^q-$estimates for the Neumann-heat semigroup in the spirit of  \cite[Lemma 1.3]{Winkler2010JDE}, see also \cite[Lemma 2.6]{li2016} for a more general version of the following estimate, we have that, for all $t\in (0,T)$, there exist two positive constants $k_1$ and $k_2$ such that
\begin{align*}
   \|\nabla\ce(.,t)\|_{L^1(\Omega)}\leq&   k_1\|\nabla \ce_0\|_{L^1(\Omega)}\\&+k_2\int_0^t\left[\left((1+(t-\tau)^{-\frac{1}{2}}\right)e^{-\lambda_1(t-\tau)}\left(\|\me(\cdot,\tau)\|_{L^1(\Omega)}+\|\de(\cdot,\tau))\|_{L^1(\Omega)}\right)\right]d\tau\\
    & \leq k_1\|\nabla c_0\|_{L^1(\Omega)}+k_2\int_0^t\left[\left((1+(t-\tau)^{-\frac{1}{2}}\right)e^{-\lambda_1(t-\tau)}(1+k_h)K_0\right]d\tau\leq k_3,
\end{align*}
where we used Lemmas \ref{boundlem} and \ref{de1}. The above estimate, together with the regularity in  Proposition \ref{Prop2} ensures that $\ce$ is an eligible function for Lemma \ref{lemm:boundary}. Thus, there exists a constant $\tilde{C}:=\tilde{C}(q)>0$ such that
\[
\int_{\partial\Omega}|\nabla \ce|^{2q-2}\frac{\partial |\nabla \ce|^{2}}{\partial \nu} \dd\sigma \leq \frac{q-1}{q^2}\into |\nabla|\nabla \ce|^q|^2 \dx +\tilde{C},
\]
and using the identity $|\nabla \ce|^{2q-4}|\nabla|\nabla \ce|^2|^2=\frac{4}{q^2}|\nabla|\nabla\ce|^q|^2$ we get \eqref{stima_2}.
\end{proof}

\begin{proof}[Proof of Lemma \ref{lemm:primo_incubo}]
We consider first the case $n\geq2$. Let $\theta>0$ be a fixed small constant such that the following Holder exponents
\[
\alpha=\frac{nq}{nq-n+2-\theta}\quad\mbox{and}\quad \beta=\frac{nq}{n-2+\theta},
\]
remain strictly greater than $1$. Then we have
\begin{equation*}
     \into \me^{p-\gamma+1}|\nabla\ce|^2\dx \leq \left(\into \me^{(p-\gamma+1)\alpha}\dx\right)^{\frac{1}{\alpha}}\left(\into |\nabla\ce|^{2\beta}\dx\right)^{\frac{1}{\beta}}.
\end{equation*}
Invoking the Poinacré's inequality, in the form of Lemma \ref{lemm:poin}, and the Sobolev embedding
\[ W^{1,2}(\Omega)\hookrightarrow L^{\frac{2\beta}{q}}(\Omega),\]
we can perform the following estimates
\begin{equation}\label{primo_incubo}
\begin{aligned}
    \left(\into |\nabla\ce|^{2\beta}\dx\right)^{\frac{1}{\beta}} & = \||\nabla\ce|^q\|_{L^{\frac{2\beta}{q}}(\Omega)}^{\frac{2}{q}}\leq k_1\||\nabla\ce|^q\|_{W^{1,2}(\Omega)}^{\frac{2}{q}}\\
    & \leq k_2\left(\|\nabla|\nabla\ce|^q\|_{L^2(\Omega)}^{\frac{2}{q}}+ \||\nabla\ce|^q\|_{L^{\frac{s}{q}}(\Omega)}^{\frac{2}{q}}\right)\\
    &\leq k_2\|\nabla|\nabla\ce|^q\|_{L^2(\Omega)}^{\frac{2}{q}}+k_3,
\end{aligned}
\end{equation}
where $s\in\left[1,\frac{n}{n-1}\right)$ and the last inequality holds because of the regularity in Proposition \ref{Prop2}. The term involving $\me$ can be bounded by using the Gagliardo-Niremberg inequality in Lemma \ref{Gag_Nir} with the choices 
\begin{equation*}
    R=2,\quad Q=S=\frac{2}{p+\gamma-1}\quad\mbox{ and }\quad P=2\alpha\frac{p-\gamma+1}{p+\gamma-1},
\end{equation*}
which are admissible thanks to the smallness of $\theta$ and the first of \eqref{condpq_1}. We recall that assumption \eqref{condition_Gag_Nir} of Lemma \ref{Gag_Nir} is satisfied for $n=2,3$. Then we can compute
\begin{equation}\label{secondo_incubo}
\begin{aligned}
    \left(\into \me^{(p-\gamma+1)\alpha}\dx\right)^{\frac{1}{\alpha}} & \leq k_4\left( \bigl\|\nabla \me ^{\frac{p+\gamma-1}{2}}\bigr\|_{L^2(\Omega)}^{\frac{2a(p-\gamma+1)}{p+\gamma-1}}\bigl\|\me^{\frac{p+\gamma-1}{2}}\bigr\|_{L^{\frac{2}{p+\gamma-1}}(\Omega)}^{\frac{2(1-a)(p-\gamma+1)}{p+\gamma-1}}+\bigl\|\me^{\frac{p+\gamma-1}{2}}\bigr\|_{L^{\frac{2}{p+\gamma-1}}(\Omega)}^{\frac{2(p-\gamma+1)}{p+\gamma-1}}\right)\\
    & \leq k_4\left( \bigl\|\nabla \me ^{\frac{p+\gamma-1}{2}}\bigr\|_{L^2(\Omega)}^{\frac{2a(p-\gamma+1)}{p+\gamma-1}}\|\me\|_{L^1(\Omega)}^{(1-a)(p-\gamma+1)}+\|\me\|_{L^1(\Omega)}^{p-\gamma+1}\right)\\
    &\leq k_5 \|\nabla \me ^{\frac{p+\gamma-1}{2}}\|_{L^2(\Omega)}^{\frac{2a(p-\gamma+1)}{p+\gamma-1}}+k_6,
\end{aligned}
\end{equation}
with 
\[
 a=\frac{\frac{p+\gamma-1}{2}-\frac{(p+\gamma-1)(nq-n+2-\theta)}{2nq(p-\gamma+1)}}{\frac{p+\gamma-1}{2}+\frac{1}{n}-\frac{1}{2}},
\]
and the last inequality holds because of Lemma \ref{boundlem}. Note that the previous definition for the constant $a$, together with  \eqref{condpq_1} and the smallness of $\theta$ ensure that
\[
 \frac{a(p-\gamma+1)}{p+\gamma-1}<\frac{q-1}{q},
\] 
thus, for any $\eta>0$, we can merge together \eqref{primo_incubo} and \eqref{secondo_incubo} and, using twice the Young's inequality, we get
\begin{equation*}
     \into \me^{p-\gamma+1}|\nabla\ce|^2\dx \leq \eta\into \bigl|\nabla \me^{\frac{p+\gamma-1}{2}}\bigr|^2\dx+\eta\into \bigl|\nabla |\nabla \ce|^q\bigr|^2\dx+C(\eta).
\end{equation*}

In order to conclude the proof we need to tackle the one-dimensional case. In this case, using the Holder inequality on the l.h.s. of \eqref{stima_4} with exponents $q$ and $\frac{q}{q-1}$ we can easily reproduce an analogue  of \eqref{primo_incubo}. On the other hand, using again Lemma \ref{Gag_Nir}
\begin{align*}
    \left(\into \me^{(p-\gamma+1)\frac{q}{q-1}}\dx\right)^{\frac{q-1}{q}} & \leq k_7\left( \bigl\|\partial_x \me ^{\frac{p+\gamma-1}{2}}\bigr\|_{L^2(\Omega)}^{\frac{2\tilde{a}(p-\gamma+1)}{p+\gamma-1}}\bigl\|\me^{\frac{p+\gamma-1}{2}}\bigr\|_{L^{\frac{2}{p+\gamma-1}}(\Omega)}^{\frac{2(1-\tilde{a})(p-\gamma+1)}{p+\gamma-1}}+\bigl\|\me^{\frac{p+\gamma-1}{2}}\bigr\|_{L^{\frac{2}{p+\gamma-1}}(\Omega)}^{\frac{2(p-\gamma+1)}{p+\gamma-1}}\right)\\
    &\leq k_8 \|\partial_x \me ^{\frac{p+\gamma-1}{2}}\|_{L^2(\Omega)}^{\frac{2\tilde{a}(p-\gamma+1)}{p+\gamma-1}}+k_9,
\end{align*}
with
\[
 \tilde{a} = \frac{\frac{p+\gamma-1}{2}-\frac{(p+\gamma-1)(q-1)}{2q(p-\gamma+1)}}{\frac{p+\gamma-1}{2}+\frac{1}{2}}=\frac{(q(p-\gamma)+1)(p+\gamma-1)}{q(p+\gamma)(p-\gamma+1)}.
\]
Then, \eqref{stima_4} follows as in the multidimensional case thanks to \eqref{condpq_2}.
\end{proof}

\begin{proof}[Proof of Lemma \ref{lemm:secondo_incubo}]
We start dealing with the case $n\geq 2$. By expanding the square we are in the position of having to estimate three integrals separately:
\begin{align*}
    A_1:=  \into \me^2 |\nabla \ce|^{2q-2}&\dx,\quad
    A_2:= 2\into \me\de |\nabla \ce|^{2q-2}\dx,\\
    A_3:= & \into \de^2 |\nabla \ce|^{2q-2}\dx.
\end{align*}
Fix $0<\theta<2$ and for each of the above integrals we perform an Holder's inequality with exponents
\[
\alpha = \frac{nq}{2q+n-2-\theta(q-1)} \quad\mbox{ and }\quad\beta = \frac{nq}{(q-1)(n-2+\theta)}.
\]
Concerning $A_1$ we have
\begin{equation*}
    A_1 \leq \left(\into \me^{2\alpha}\dx\right)^{\frac{1}{\alpha}}\left(\into |\nabla \ce|^{2(q-1)\beta}\dx\right)^\frac{1}{\beta}.
\end{equation*}
Similarly to what we did in Lemma \ref{lemm:primo_incubo}, we can use Sobolev embedding, Poincaré's inequalities and Proposition \ref{Prop2} to deduce the bound
\begin{equation}\label{terzo_incubo}
    \begin{aligned}
        \left(\into |\nabla \ce|^{2(q-1)\beta}\dx\right)^\frac{1}{\beta} & =\bigl\||\nabla \ce|^q\bigr\|_{L^{\frac{2n}{n-2+\theta}}(\Omega)}^{\frac{2(q-1)}{q}}\leq k_1\bigl\||\nabla \ce|^q\bigr\|_{W^{1,2}(\Omega)}^{\frac{2(q-1)}{q}}\\
        &\leq k_2 \left(\bigl\|\nabla|\nabla \ce|^q\bigr\|_{L^{2}(\Omega)}^{\frac{2(q-1)}{q}}+\bigl\||\nabla \ce|^q\bigr\|_{L^{\frac{s}{q}}(\Omega)}^{\frac{2(q-1)}{q}}\right)\\
           &\leq k_2 \left(\bigl\|\nabla|\nabla \ce|^q\bigr\|_{L^{2}(\Omega)}^{\frac{2(q-1)}{q}}+k_3\right)
    \end{aligned}
\end{equation}
with $s\in\left[1,\frac{n}{n-1}\right)$ and for all $t\in \Tin$. Note that \eqref{terzo_incubo} holds true for each of the three estimates we are going to perform. Taking $\theta$ small enough and applying the Gagliardo-Niremberg inequality with exponents
\[
P=\frac{4nq}{(p+\gamma-1)(2q+n-2-\theta(q-1))}, \quad R=2, \quad Q=\frac{2}{p+\gamma-1}=S,
\]
and
\[
 \bar{a}_1 =\frac{\frac{p+\gamma-1}{2}-\frac{(p+\gamma-1)(2q+n-2-\theta(q-1))}{4nq}}{\frac{p+\gamma-1}{2}+\frac{1}{n}-\frac{1}{2}},
\]
we have
\begin{equation}\label{quarto_incubo}
    \begin{aligned}
        \left(\into \me^{2\alpha}\dx\right)^{\frac{1}{\alpha}} &= \bigl\|\me^{\frac{p+\gamma-1}{2}}\bigr\|_{L^{P}(\Omega)}^{\frac{4}{p+\gamma-1}}\\
        &\leq k_4\left(\bigl\|\nabla\me^{\frac{p+\gamma-1}{2}}\bigr\|_{L^{2}(\Omega)}^{\frac{4\bar{a}_1}{p+\gamma-1}}\bigl\|\me^{\frac{p+\gamma-1}{2}}\bigr\|_{L^{\bar{q}}(\Omega)}^{\frac{4(1-\bar{a}_1)}{p+\gamma-1}}+\bigl\|\me^{\frac{p+\gamma-1}{2}}\bigr\|_{L^{\bar{s}}(\Omega)}^{\frac{4}{p+\gamma-1}}\right)\\
        &\leq k_5\left(\bigl\|\nabla\me^{\frac{p+\gamma-1}{2}}\bigr\|_{L^{2}(\Omega)}^{\frac{4\bar{a}_1}{p+\gamma-1}}+k_6\right).
    \end{aligned}
\end{equation}
where in the last inequality we used the bound on the $L^1-$norm in  Lemma \ref{boundlem}. Thus, performing twice the Young's inequality and invoking condition \eqref{condpq_3} that ensure 
\[
\frac{2\bar{a}_1}{p+\gamma-1}<\frac{1}{q},
\]
we can conclude that, given $\eta>0$
\begin{equation}\label{terzo e quarto}
    \begin{aligned}
        A_1 &\leq  k_2 k_5\left(\bigl\|\nabla\me^{\frac{p+\gamma-1}{2}}\bigr\|_{L^{2}(\Omega)}^{\frac{4\bar{a}_1}{p+\gamma-1}}+k_6\right) \left(\bigl\|\nabla|\nabla \ce|^q\bigr\|_{L^{2}(\Omega)}^{\frac{2(q-1)}{q}}+k_3\right)\\
        & \leq  \eta\bigl\|\nabla\me^{\frac{p+\gamma-1}{2}}\bigr\|_{L^{2}(\Omega)}^2+\eta\bigl\|\nabla|\nabla \ce|^q\bigr\|_{L^{2}(\Omega)}^2+k_7
    \end{aligned}
\end{equation}
We turn now to the estimate of $A_2$. As already mentioned, the term involving $\ce$ can be treated as in \eqref{terzo_incubo}. Thus, using Holder and Young inequalities we have
\begin{equation}\label{quinto_incubo}
    \begin{aligned}
          2\left(\into (\me\de)^{\alpha}\dx\right)^{\frac{1}{\alpha}} & \leq 2\left(\into \me^{2\alpha}\dx\right)^{\frac{1}{2\alpha}}\left(\into \de^{2\alpha}\dx\right)^{\frac{1}{2\alpha}}\\
           &\leq\left(\into \me^{2\alpha}\dx\right)^{\frac{1}{\alpha}}+\left(\into \de^{2\alpha}\dx\right)^{\frac{1}{\alpha}}\\
            &\leq k_5\left(\bigl\|\nabla\me^{\frac{p+\gamma-1}{2}}\bigr\|_{L^{2}(\Omega)}^{\frac{4\bar{a}_1}{p+\gamma-1}}+k_6\right)+k_8,
    \end{aligned}
\end{equation}
where the last inequality holds because of \eqref{quarto_incubo} and Lemma \ref{de_sign}. 
The term $A_3$ can be estimated similarly to $A_1$, indeed
\begin{equation*}
    A_3 \leq \left(\into \de^{2\alpha}\dx\right)^{\frac{1}{\alpha}}\left(\into |\nabla \ce|^{2(q-1)\beta}\dx\right)^\frac{1}{\beta}\leq k_8\left(\into |\nabla \ce|^{2(q-1)\beta}\dx\right)^\frac{1}{\beta},
\end{equation*}
and then applying \eqref{terzo_incubo}. Thus \eqref{stima_5} follows with several application of the Young's inequality similarly to what we did in \eqref{terzo e quarto}.

In the one-dimensional case we perform a Holder's inequality with
\[
\alpha = q \quad\mbox{ and } \quad \beta = \frac{q}{q-1},
\]
then the equivalent of \eqref{terzo_incubo} is straightforward and we have
\begin{equation}\label{quarto_incubo_bis}
    \begin{aligned}
        \left(\into \me^{2q}\dx\right)^{\frac{1}{q}} &= \bigl\|\me^{\frac{p+\gamma-1}{2}}\bigr\|_{L^{\frac{4q}{p+\gamma-1}}(\Omega)}^{\frac{4}{p+\gamma-1}}
        \leq k_{9}\left(\bigl\|\partial_x\me^{\frac{p+\gamma-1}{2}}\bigr\|_{L^{2}(\Omega)}^{\frac{4\tilde{a}_1}{p+\gamma-1}}+k_{10}\right),
    \end{aligned}
\end{equation}
where the exponent $\tilde{a}_1$ is given by
\[
\tilde{a}_1=\frac{\frac{p+\gamma-1}{2}-\frac{p+\gamma-1}{4q}}{\frac{p+\gamma-1}{2}+\frac{1}{2}}.
\]
Thus, condition \eqref{condpq_4} allows to perform Young's inequality in the spirit of what we did in the multidimensional case. 
\end{proof}

\section{A finite-volume numerical scheme}\label{sec:appendix num}
The method we use is a modification of the finite volume method introduced in \cite{CCH} and extended to systems in \cite{CHS}, that consists in a positivity-preserving finite-volume method for system \eqref{eq:main_PM}. We sketch the scheme the one-dimensional setting. We consider a partition for the computational domain into finite-volume cells $U_{i}=\left[x_{i-\frac{1}{2}},x_{i+\frac{1}{2}}\right]$ of a uniform size $\Delta x$ with $x_{i} = i\Delta x$, $i \in\{-s,~\dots,s \}$, and we define
\begin{equation*}
\widetilde{m}_{i}(t) := \frac{1}{\Delta x}\int_{U_{i}}m(x,t)\dx,\qquad \widetilde{c}_{i}(t) := \frac{1}{\Delta x}\int_{U_{i}}c(x,t)\dx,\qquad \widetilde{d}_{i}(t) := \frac{1}{\Delta x}\int_{U_{i}}d(x,t)\dx,
\end{equation*}
the averages of the solutions $m$, $c$ and $d$ computed at each cell $U_{i}$. Those averages are obtained as solutions of the semi-discrete scheme described by the following system of ODEs 
\begin{subequations}\label{volume ode1}
\begin{align}
\frac{\dd\widetilde{m}_{i}(t)}{\dt}& = -\frac{F_{i+\frac{1}{2}}^{m}(t) - F_{i-\frac{1}{2}}^{m}(t) }{\Delta x}+\widetilde{m}_{i}(t)(1-\widetilde{m}_{i}(t)),
\\
\frac{\dd \widetilde{c}_{i}(t)}{\dt} &= -\frac{F_{i+\frac{1}{2}}^{c}(t) - F_{i-\frac{1}{2}}^{c}(t) }{\Delta x}+\lambda \widetilde{d}_i(t) - \widetilde{c}_{i}(t) +\beta \widetilde{m}_{i}(t),
\\
\frac{\dd\widetilde{d}_{i}(t)}{\dt} &= r\widetilde{m}_{i}(t)\frac{\widetilde{m}_{i}(t)}{1+\delta \widetilde{m}_{i}(t)}\left(1-\widetilde{d}_{i}(t)\right),
\end{align}
\end{subequations}
where the numerical fluxes $F_{i+\frac{1}{2}}^l$, $l=m,c$, are considered as an approximation of the transport parts in equation \eqref{eq:main_PM}. More precisely, defining
\begin{equation}\label{minmod}
\begin{split}
(m_{x})^{i} & = \mbox{minmod} \Bigg( \theta\frac{\widetilde{m}_{i+1} - \widetilde{m}_{i}}{\Delta x},~ \frac{\widetilde{m}_{i+1} - \widetilde{m}_{i-1}}{2\Delta x},~ \theta\frac{\widetilde{m}_{i} - \widetilde{m}_{i-1}}{\Delta x} \Bigg),\\
(c_{x})^{i}&  = \mbox{minmod} \Bigg( \theta\frac{\widetilde{c}_{i+1} - \widetilde{c}_{i}}{\Delta x},~ \frac{\widetilde{c}_{i+1} - \widetilde{c}_{i-1}}{2\Delta x},~ \theta\frac{\widetilde{c}_{i} - \widetilde{c}_{i-1}}{\Delta x} \Bigg).
\end{split}
\end{equation}
with $\theta\in[0,2]$, where the minmod limiter in \eqref{minmod} has the following definition
\begin{equation*}
\mbox{minmod}(a_{1}, a_{2},~\dots) :=
\begin{cases}
\min (a_{1}, a_{2},~\dots), \quad \mbox{if} ~ a_{i} > 0\quad \forall i
\\
\max (a_{1}, a_{2},~\dots), \quad \mbox{if} ~ a_{i} < 0\quad \forall i
\\
0, \qquad\qquad\qquad\mbox{otherwise.}
\end{cases}
\end{equation*}
and 
\begin{equation}
\vartheta_{m}^{i+1} =  - \frac{\gamma}{(\gamma-1)\Delta x} \big( \widetilde{m}_{i+1}^{\gamma-1} - \widetilde{m}_{i}^{\gamma-1} \big) + \chi\frac{(c_x)^i}{1+\widetilde{m}_{i}},
\end{equation}
the expression for $F_{i+\frac{1}{2}}^{m}$ is given by
\begin{equation}
F_{i+\frac{1}{2}}^{m} = \max (\vartheta_{}^{i+1},0)\Big[ \widetilde{m}_{i} + \frac{\Delta x}{2}(m_{x})_{i} \Big] + \min (\vartheta_{m}^{i+1},0)\Big[ \widetilde{m}_{i} - \frac{\Delta x}{2}(m_{x})^{i} \Big],
\end{equation}
Similarly we can recover the numerical flux for $c$. Finally, we integrate the semi-discrete scheme \eqref{volume ode1} numerically using the third-order strong preserving Runge-Kutta (SSP-RK) ODE solver used in \cite{gott}.

\end{appendices}

\end{document}